\documentclass[twoside,a4paper,titlepage,10pt]{article}

\usepackage[symbol]{footmisc}
\usepackage{epigraph}
\usepackage{geometry}
\usepackage{amsmath}
\usepackage{amsfonts}
\usepackage{amssymb}
\usepackage{amsthm}
\usepackage{mathrsfs}
\usepackage{mathtools}
\usepackage{esint}
\usepackage{graphicx}
\usepackage{color}
\usepackage{hyperref}
\usepackage{tikz}
\usepackage{float}
\usepackage{mathrsfs}
\usetikzlibrary{patterns}
\usepackage{enumitem}
\usepackage[english]{babel}

\newcommand{\mc}[1]{\mathcal{#1}}
\newcommand{\LO}{{\Lambda_\Omega}}
\newcommand{\tLO}{\tilde{\Lambda}_\Omega}

\newcommand{\WL}{{\mathcal{W}_\Lambda}}

\newcommand{\res}{
	\,\raisebox{-.127ex}{\reflectbox{\rotatebox[origin=br]{-90}{$\lnot$}}}\,
} 

\newcommand{\sm}{\setminus}

\newcommand{\lgl}{\langle}
\newcommand{\rgl}{\rangle}

\newcommand{\supp}{\operatorname{supp}}

\newcommand{\pa}{\partial}

\newcommand{\con}{\subset} 

\newcommand{\R}{\mathbb R}
\newcommand{\N}{\mathbb N}

\newcommand{\V}{\mathbb V}

\newcommand{\W}{\mathbb W}

\newcommand{\bF}{\mathbf{F}}

\newcommand{\bM}{\mathbf{M}}

\newcommand{\bv}{\mathbf{v}}

\newcommand{\tA}{\tilde{A}}
\newcommand{\tB}{\tilde{B}}
\newcommand{\tC}{\tilde{C}}

\newcommand{\tH}{\tilde{H}}

\newcommand{\tS}{\tilde{\Si}}

\newcommand{\HH}{\mathcal H}
\newcommand{\LL}{\mathcal L}
\newcommand{\RR}{\mathcal R}

\newcommand{\VV}{\mathcal V}
\newcommand{\WW}{\mathcal{W}}

\newcommand{\ga}{\gamma}
\newcommand{\be}{\beta}
\newcommand{\al}{\alpha}
\newcommand{\de}{\delta}
\newcommand{\ep}{\epsilon}

\newcommand{\la}{\lambda}

\newcommand{\ro}{\rho}
\newcommand{\si}{\sigma}
\newcommand{\te}{\theta}
\newcommand{\De}{\Delta}
\newcommand{\Ga}{\Gamma}
\newcommand{\La}{\Lambda}
\newcommand{\Si}{\Sigma}
\newcommand{\Te}{\Theta}
\newcommand{\Om}{\Omega}
\newcommand{\vp}{\varphi}

\newcommand{\Nau}{\mathcal{N}}

\theoremstyle{plain}
\newtheorem{thm}{Theorem}[section] 
\newtheorem*{thm*}{Theorem A}
\newtheorem*{thm**}{Theorem B}

\theoremstyle{plain}

\theoremstyle{plain}
\newtheorem{prop}[thm]{Proposition}

\theoremstyle{plain}
\newtheorem{lemma}[thm]{Lemma}

\theoremstyle{plain}
\newtheorem{cor}[thm]{Corollary}

\theoremstyle{definition}
\newtheorem{defn}[thm]{Definition} 

\theoremstyle{definition}
\newtheorem{remark}[thm]{Remark}

\theoremstyle{definition}
\newtheorem{example}[thm]{Example}

\begin{document}

	\begin{center}
		\LARGE\textbf{Confined Willmore energy and the Area functional}
	\end{center}
	\vspace{-3mm}
	\begin{center}
	By \scshape Marco Pozzetta\footnote{
		\noindent pozzetta@mail.dm.unipi.it, Dipartimento di Matematica, Universit\`{a} di Pisa, Largo Bruno Pontecorvo 5, 56127 Pisa, Italy.
		}
		\vspace{4mm}\\
		\today
	\end{center}
	\begin{minipage}[h]{\textwidth}
	\begin{abstract}
	We consider minimization problems of functionals given by the difference between the Willmore functional of a closed surface and its area, when the latter is multiplied by a positive constant weight $\Lambda$ and when the surfaces are confined in the closure of a bounded open set $\Omega\subset\mathbb{R}^3$. We explicitly solve the minimization problem in the case $\Omega=B_1$. We give a description of the value of the infima and of the convergence of minimizing sequences to integer rectifiable varifolds, depending on the parameter $\Lambda$. We also analyze some properties of these functionals and we provide some examples. Finally we prove the existence of a $C^{1,\alpha}\cap W^{2,2}$ embedded surface that is also $C^\infty$ inside $\Omega$ and such that it achieves the infimum of the problem when the weight $\Lambda$ is sufficiently small.
	\end{abstract}
\end{minipage}

\noindent\textbf{MSC Codes:}  53A05, 49Q15.

\tableofcontents


\section*{Introduction}

\addcontentsline{toc}{section}{Introduction}

If $\Si\con\R^3$ is a smooth immersed surface and $H$ is its mean curvature vector, that we define with norm equal to the absolute value of the arithmetic mean of the principal curvatures, we define the Willmore energy of $\Si$ as:
\begin{equation}
\WW(\Si):=\int_{\Si} |H|^2\,d\si
\end{equation}
where $\si$ is the area measure on $\Si$. The opertor $\WW$ is called Willmore functional. Surfaces will be usually denoted by $\Si$ and will be always compact and without boundary, but not necessarily connected.\\
The variational study of this functional has been revived in 1965 with the work of T. Willmore (\cite{Wi65} and \cite{WiRG}). He found that round spheres are the only global minimizers for $\WW$ and then he introduced the study of the minimization problem subject to constraints of topological type, such as fixing the genus of the surfaces; the celebrated Willmore Conjecture is related to these kind of problems, and it has been proved in \cite{MaNeWC}. In the last decades a number of properties about the functional itself have been proved, and the ones we will use are recalled in Section 1. The minimization problem at fixed genus has also been solved in a couple of works (\cite{SiEX} and then \cite{BaKu}), developing also a theory of which we will make use in the following.\\
In this work we are going to study the following functional:
\begin{equation}
\WL(\Si):=\WW(\Si)-\La|\Si|,
\end{equation}
where $\Si\con\R^3$ is a smooth surface, $\La>0$ is fixed and $|\Si|$ denotes the area of $\Si$. Moreover, we will always consider surfaces $\Si\con\bar{\Om}$ with $\Om\con\R^3$ open and bounded with $\pa\Om$ of class $C^2$. Also, by a rescaling property shown in Section 1, we will usually take $\Om\con B_{\frac{1}{2}}(0)$ (so that ${\rm diam}(\Om)\le1$).\\
With the above assumptions we show that the minimization problem
\begin{equation}
(P)_{\Om,\La}:\qquad\min \{ \WL(\Si):\Si\con \bar{\Om} \}
\end{equation}
sets a non trivial competition between the Willmore and the Area terms. We define
\begin{equation}
C_\La:=\inf \{ \WL(\Si):\Si\con \bar{\Om} \}.
\end{equation}
We also give here the following definitions, that will be useful later on:
\begin{equation} 
\LO:=\inf\{\La>0: C_\La =-\infty \}, \qquad
\W(\Si):=\frac{\WW(\Si)}{|\Si|},\qquad \tLO:=\inf\{ \W(\Si):\Si\con\bar{\Om} \}.
\end{equation}
Variational problems of a similar type, that is problems involving the area, have already been studied. There is a complete treatment of the minimization problem of the Willmore energy with fixed area for surfaces of genus zero (\cite{MoRi} and \cite{MuRoCS}) and with fixed isoperimetric ratio for surfaces of arbitrary fixed genus (\cite{Schy} and \cite{KeMoRi}). This kind of works found interesting comparisons with works about the shape of organic corpuscles (\cite{SeCMV}).The link with the quantities $\W$ and $\tLO$ defined above resembles the Cheeger Problem, which is actually strongly related to the existence of confined surfaces with prescribed mean curvature vector (\cite{LeCP}). It would be interesting to study related problems for curves in dimension two, for which there are already remarkable results about the variational problems of functionals depending on the curvature of the curve in the same way the Willmore energy depends on the curvature of the surface (in \cite{DMN} and \cite{DMR} confined elastic curves are considered, while in \cite{BeMu04} \cite{BeMu07} and \cite{Po20} relaxed notions of the elastic energy are studied).\\
In the next statement we sum up our main results in the case of a general domain $\Om$.
\begin{thm*}\label{thmmainresults}
	Under the above assumptions on $\Om$, denoting $C_\La:(0,+\infty)\to[-\infty,+\infty)$ the function that associates to $\La$ the infimum of $(P)_{\Om,\La}$, it holds:
	\begin{enumerate} [label={\normalfont(\roman*)}]
	\item $C_\La$ is a concave, continuous, non negative, strictly decreasing function on an interval $(0,\LO]$ for some $\LO\in[4,1/\ep_\Om^2]$ where $\LO, \ep_\Om$ depend on $\Om$. Moreover $\lim_{\La\to0^+}C_\La=4\pi$, $C_\LO\ge 0$ and $C_\La=-\infty$ for all $\La>\LO$.
	\item If $\La\in(0,\LO)$ there exists a sequence $(\Si_n^\La)$ that is minimizing for the functional $\WL$ and such that it converges in the sense of varifolds to a varifold $V$ that is integer rectifiable and has generalized mean curvature square integrable with respect to the weight measure of $V$.
	\item If $\La$ is sufficiently small, depending only on $\Om$, the limit varifold in item (ii) is actually a $C^{1,\al}\cap W^{2,2}$ embedded surface $\Si$ with multiplicity one and it is such that $\WL(\Si)=C_\La$. Moreover it holds that $\Si\cap\Om$ is of class $C^\infty$.
\end{enumerate}
\end{thm*}
\noindent Next we state the result concerning the case $\Om=B_1$, where $B_1$ is the standard unit ball of $\R^3$.
\begin{thm**}\label{thmmainresultsball}
	If $\Om=B_1$ the minimization problem $(P)_{B_1,\La}$ admits a solution if and only if $\La\le1$, in which case the minimum is $4\pi(1-\La)$. If $\La<1$ the unique minimizer is the unit sphere $S^2$. Moreover for all $\La>1$ the infimum of the problem is $-\infty$.
\end{thm**}

\noindent The paper is organized as follows. In Section 1 we state some classical properties of the Willmore functional and of the $\WL$ energy. In Section 2 we prove the first two items of Theorem A and we prove Theorem B. Section 3 is devoted to the proof of item $(iii)$ of Theorem A, that is essentially a regularity issue. In this work we adopt a very classical method, today named \emph{Simon's ambient approach} (\cite{SiEX}), that is well applicable in our setting. We will mainly highlight the differences that arise with respect to \cite{SiEX}, that is taking care of the area term and of the presence of the boundary $\pa\Om$. By now we just mention that this method is based on the direct proof of the regularity of a set contained in $\R^3$ from information about the boundedness of its second fundamental form and it has already been used in other works linked to the Willmore energy (\cite{KuSc}, \cite{Mi}, \cite{ScWBP} and \cite{Schy}). It is very remarkable a more modern method, called \emph{parametrization approach}, essentially due to Rivi\`{e}re and presented for example in \cite{RiAA}, \cite{RiLI} and \cite{RiVP}. This method is based of the formulation of suitable spaces of parametrizations of surfaces, where abstract techniques of calculus of variations are applicable. Notable applications are contained in the already cited \cite{KeMoRi} and \cite{MoRi}. We think that applying this method to our problem could give very good results and it can certainly be a future project to improve our current results following this way.\\

\noindent \emph{Acknowledgments:} I am very grateful to Matteo Novaga for his help and his interest during the preparation of this work, that is partially contained in my master thesis. I also thank Giovanni Alberti for some precious observations.

\section{Basic Properties}

We are going to collect some useful properties about the Willmore functional that we will use later on. The symbol $\VV_2 (\bar{\Om})$ denotes the set of 2-rectifiable integer varifold defined in $\R^3$ with support contained in $\bar{\Om}$. The convergence in $\VV_2 (\bar{\Om})$ is the classical convergence of varifolds in $\R^3$. The symbol $\mu_V$ will always denote the Radon measure on $\bar{\Om}$ induced by the varifold $V\in\VV_2(\bar{\Om})$. We recall that $\Om\con\R^3$ is open, bounded and with $\pa\Om$ of class $C^2$. For the general notation and results about varifolds see Appendix A. Let us start with an important observation.
\begin{remark}[Semicontinuity] \label{remsemicontinuity}
	Let us consider a sequence $V_k\in \VV_2(\bar{\Om})$ that converges to $V\in\VV_2(\bar{\Om})$ in the sense of varifolds. Assume that for each $k$ there exists the generalized mean curvature $H_k$ of $V_k$ such that
	\begin{displaymath}
	\WW(V_k):=||H_k||^2_{L^2(\bar{\Om},\mu_{V_k})}\le C_0<+\infty,
	\end{displaymath}
	with $C_0$ independent of $k$.\\
	Then applying convergence of Radon measures and using the continuity of the first variation with respect to the varifold convergence, we have that $V$ has generalized mean curvature $H_V$ such that
	\begin{displaymath}
	\liminf_k \WW(V_k)\ge \WW(V).
	\end{displaymath}
	We note also that, since $\bar{\Om}$ is compact, we have that $\bM(V_k)\to\bM(V)$, where $\bM$ denotes the mass of a varifold. Thus:
	\begin{displaymath}
	\liminf_k \WL(V_k)\ge \WL(V),
	\end{displaymath}
	\begin{displaymath}
	\liminf_k \W(V_k)\ge \W(V).
	\end{displaymath}
	For further details see \cite{ScLSC}, where it is also shown the more involved lower semicontinuity property under convergence of currents.
\end{remark}

\noindent Now we state a couple of fundamental properties of the Willmore energy.
\begin{thm}[Conformal Invariance, \cite{WiRG}]\label{thm:ConfInv}
	Let $\Si\con\R^3$ be an immersed surface in the 3-dimensional Euclidean space. Suppose $\Si\con\Om$, with $\Om\con\R^3$ open. If $F:\Om\to F(\Om)$ is a conformal transformation, then:
	\begin{displaymath}
	\WW(\Si)=\WW(F(\Si)).
	\end{displaymath}
\end{thm}
\begin{remark}
	We recall that, by Liouville's Theorem, conformal transformations of the Euclidean $\R^3$ are just compositions of translations, dilatations, orthogonal transformations and spherical inversions (for an interesting proof see \cite{PhLT}).
\end{remark}

\begin{thm}[Lower Bound for Immersed Surfaces, \cite{BaKu}]\label{thmlowerbound}
	Let $\Si$ be an immersed surface and $\xi\in\Si$ be a point with multiplicity $k$. If $I:\R^3\sm\{\xi\}\to \R^3\sm\{\xi\}$ is the standard spherical inversion about the sphere $S^2_1(\xi)$, then:
	\begin{equation}
	\WW(I(\Si\sm\{\xi\}))=\WW(\Si)-4\pi k.
	\end{equation}
\end{thm}
\begin{remark} \label{remmoltvar}
	We immediately get from Theorem \ref{thmlowerbound} that if a surface $\Si$ has a point with multiplicity $k$, then $\WW(\Si)\ge4\pi k$.\\
	A similar argument holds for a varifold $V\in\VV(\bar{\Om})$ with square integrable generalized mean curvature in the sense that it holds:
	\begin{equation}
	\te(x)\le\frac{\WW(V)}{4\pi} \qquad \mu_V\mbox{-almost every }x,
	\end{equation}
	where $\te$ is the multiplicity function of $V$ and $\mu_V$ is the Radon measure given by $V$ on $\R^3$ (see \cite{KuSc}, Appendix A). In particular we get that if $\WW(V)<8\pi$, then the varifold has multiplicity $1$ $\mu_V$-almost everywhere.
\end{remark}

\noindent Now we state some results relating the Willmore and the Area functionals.
\begin{lemma}[\cite{SiEX}] \label{lemdisugsimon}
	If $\Si\con\R^3$ is a connected surface, then:
	\begin{displaymath}
	\sqrt{\frac{|\Si|}{\WW(\Si)}}\le {\rm diam}\Si \le \frac2\pi\sqrt{|\Si|\WW(\Si)},
	\end{displaymath}
	where ${\rm diam}\Si$ is the diameter of $\Si$ and $C$ is a constant independent of $\Si$.
\end{lemma}
\noindent Another fundamental inequality is the following:
\begin{lemma}[Willmore vs Area Inequality, \cite{MuRoCS}] \label{thmwillmorevsareainequality}
	Let $\Om\con B_1\con\R^3$ and let $V\in\VV_2(\bar{\Om})$ such that there exists the generalized mean curvature $H_V\in L^2(\bar{\Om},\mu_V)$. Then:
	\begin{equation} \label{eqwillmorevsareainequality}
	\WW(V):=\int_\Om H_V^2 \, d\mu_V \ge \bM(V),
	\end{equation}
	with equality if and only if $\mu_V=k\HH^2\res S^2$ and $S^2\con\bar{\Om}$, with $k\in \N_{>0}$.
\end{lemma}
\begin{remark} \label{rmkwillmorevsarea2}
	The inequality proved in Lemma \ref{thmwillmorevsareainequality} can be specialized in the case of $\Om\con B_{\frac{1}{2}}$. If $V\in \VV_2(\bar{\Om})$, by a simple scaling argument and using the conformal invariance of $\WW$ one gets that
	\begin{equation}
		\WW(V)\ge 4 \bM(V).
	\end{equation}
\end{remark}
\noindent From these results we establish some very useful inequalities, as stated in the following.
\begin{cor} \label{cordisugutili}
	If $\Si\con\bar{\Om}$ is a connected surface, then
	\begin{displaymath}
	\begin{split}
	&\WL(\Si)\ge |\Si|\bigg( \frac{1}{({\rm diam}\Si)^2}-\La  \bigg) ,\\
	&\WL(\Si)\ge \WW(\Si)(1-\La ({\rm diam}\Si)^2), \\
	&\WL(\Si)\le \WW(\Si)-\frac{\La}{C^2}\frac{({\rm diam}\Si)^2}{\WW(\Si)}, \\
	&\W(\Si)\ge\frac{1}{({\rm diam}\Si)^2}, \\
	&\W(\Si)\ge \frac{1}{C^2} \frac{({\rm diam}\Si)^2}{|\Si|^2},
	\end{split}
	\end{displaymath}
	where $C$ is the constant in Lemma \ref{lemdisugsimon}.\\
	If $\Om\con B_1$ then
	\begin{displaymath}
	\begin{split}
	&\WL(\Si)\ge |\Si|( 1-\La  ) ,\\
	&\WL(\Si)\ge \WW(\Si)(1-\La ), \\
	&\W(\Si)\ge1.
	\end{split}
	\end{displaymath}
\end{cor}

\noindent Finally, we derive a simple but useful result about invariance under dilatation.
\begin{lemma} \label{leminvarianza}
	For all $\Si$ surface and for all $\al>0$ it holds:
	\begin{displaymath}
	\WW(\Si)=\WW(\al\Si),
	\end{displaymath}
	\begin{displaymath}
	\WL(\Si)=\mathcal{W}_{\frac{\La}{\al^2}}(\al\Si),
	\end{displaymath}
	\begin{displaymath}
	\W(\Si)=\al^2\W(\al\Si).
	\end{displaymath}
\end{lemma}
\begin{proof} The first equation is a consequence of the conformal invariance of the Willmore functional. For the same property we have:
	\begin{displaymath}
	\WL(\Si)=\left(\int_\Si H_\Si^2\right)-\La|\Si|=\left(\int_{\al\Si} H_{\alpha H}^2\right) -\La|\Si|=\WW(\alpha \Si) -\frac{\La}{\al^2}|\al\Si|=\mathcal{W}_{\frac{\La}{\al^2}}(\al\Si).
	\end{displaymath}
	By the same token we get the last equality:
	\begin{displaymath}
	\W(\al\Si)=\frac{\WW(\Si)}{|\al\Si|}=\frac{\WW(\Si)}{\al^2|\Si|}=\frac{1}{\al^2}\W(\Si).
	\end{displaymath}
\end{proof}
\begin{remark} \label{reminvarianza}
	From Lemma \ref{leminvarianza} we see that from the variational point of view we have the following equivalence of problems:
	\begin{equation}
	\begin{split}
	&(P)_{\Om,\La} \longleftrightarrow (P)_{\al\Om,\frac{\La}{\al^2}},
	\end{split}
	\end{equation}
	in the sense that if we have that for a couple $(\Om,\La)$ there exists minimum of $(P)_{\Om,\La} $ then the same holds for the couple $(\al\Om,\La/\al^2)$ and with the same value of minimum (and the same holds in case of nonexistence of minima).\\
	For these reasons in the study of Problem $(P)_{\Om,\La}$ with generic $\Om$ we will exploit this invariance assuming $\Om\con B_{1/2}$ without loss of generality.
\end{remark}

\section{Compactness and properties of $C_\La$}

\noindent This section is devoted to the proof of items $i),\, ii)$ of Theorem A and of Theorem B. Let us start with a significant example.

\begin{example} \label{exdante}
	If $\La>1$ then
	\begin{displaymath}
	\inf_{\Si\con\bar{B_1}} \WL(\Si) =-\infty.
	\end{displaymath}
	In fact let us define the sequence of surfaces
	\begin{equation}
	\begin{split}
	&D^k=S^2_{r_1}\cup \dots \cup S^2_{r_k} \con B_1, \\
	&r_i=\frac{1}{\sqrt{\La}} + \frac{i-1}{k}\bigg( 1-\frac{1}{\sqrt{\La}}   \bigg) \qquad i=1,\dots ,k,
	\end{split}
	\end{equation}
	that is a surface made of $k$ concentric spheres with minimum radius $r_1=1/\sqrt{\La}$, $r_i<r_{i+1}$ for $i=1,\dots ,k-1$ and maximum radius $r_k<1$ (since $\La>1$). We have:
	\begin{displaymath}
	\begin{split}
	\WL(D^k)&=4\pi k -\La \sum_{i=1}^{k} 4\pi r_i^2 \\
	&= 4\pi\bigg(  k-k-\La  \sum_{i=1}^{k}  \bigg( 1-\frac{1}{\sqrt{\La}}   \bigg) ^2  \frac{1}{k^2} (i-1)^2 + \frac{2}{k\sqrt{\La}}\bigg( 1-\frac{1}{\sqrt{\La}}   \bigg)(i-1)  \bigg) \\
	&\le -4\pi \La \frac{2}{\sqrt{\La}}\bigg( 1-\frac{1}{\sqrt{\La}}   \bigg) \frac{1}{k}\bigg(  \frac{k(k+1)}{2} -k \bigg) \to -\infty \qquad\qquad k\to+\infty,
	\end{split}
	\end{displaymath}
	where we strongly used the fact that $\La>1$.
\end{example}
\noindent Now we see that we can actually connect together the rounds of the previous example in a way in which we are able to obtain the same conclusion also in the case in which the problem is restricted to connected surfaces. We are going to see this in a general way as stated in next lemma.

\begin{lemma} \label{scattomenoinfinito}
	If there exists an embedded surface $\Si\con\bar{\Om}$ such that $\WL(\Si)<0$, then there exists a sequence of embedded surfaces $\Si_n$ such that $\WL(\Si_n)\to-\infty$. In particular $C_\La=-\infty$.\\
	Moreover if $\Si$ is connected, the surfaces $\Si_n$ can be taken connected.
\end{lemma}
\begin{proof}
	We are going to reproduce the idea of Example \ref{exdante} with the surface $\Si$ in the hypothesis of the statement. Let us fix $\ep\in(0,1)$. First we notice that it may occur that $\Si_c:=\Si\cap \pa \Om\neq\emptyset$, and so we suppose we are in this situation (the case $\Si_c=\emptyset$ will be a simpler by-product of this case). Let us fix for each connected component $\Si_c^\al$ (note that $\Si_c^\al$ is compact) a field $N^\al\in\Nau(\Si_c)$ such that $N^\al$ point inside $\Om$ for each $\al$, where $\Nau(\Si_c)$ denotes the normal bundle of $\Si_c$. Now fix an open neighborhood $U^\al\con\Si$ of each $\Si_c^\al$ such that $dist(p,\Si_c^\al)<\de$ for each $p\in U^\al$ and $U^\al\cap U^\be =\emptyset$ for all $\al\neq \be$. Let for all $\al$ the functions $\phi^\al\in C^\infty_c(\Si)$ such that $\phi^\al(p)=1$ for all $p\in \Si_c^\al$ and $\phi^\al(p)=0$ for all $p\in \Si\sm U^\al$. Now mapping:
	\begin{displaymath}
	U^\al \ni p \longmapsto p+\de N^\al(p)\phi^\al(p) \in \Om,
	\end{displaymath}
	we obtain a new embedded surface $\Si'\con\Om$ such that for an appropriate choice of $\de$ above sufficiently small we have:
	\begin{displaymath}
	\WL(\Si')=\WL(\Si)+\ep.
	\end{displaymath}
	Since $\Si'$ is compact and embedded, it is orientable, so there exists a field $N\in\Nau(\Si')$ that orients the surface. For $M\in\R$ sufficiently big we can consider the surface:
	\begin{displaymath}
	\Si'_M:=\bigg\{p+\frac{1}{M} N(p):p\in\Si'   \bigg\}\con\Om \qquad\mbox{s.t.}\quad\WL(\Si'_M)=\WL(\Si')+\ep.
	\end{displaymath}
	Now we are going to connect together $\Si'$ with $\Si'_M$ in order to obtain the first term $\Si_1$ of the desired sequence $(\Si_n)$. Select $\bar{p}\in\Si'$ and consider the corresponding $\bar{p}_M=\bar{p}+\frac{1}{M} N(\bar{p})\in\Si'_M$. Letting $\bar{q}$ the middle point between $\bar{p}$ and $\bar{p}_M$, there exists $\de_0$ such that $(\Si'\cup\Si'_M)\cap B_{\de_0}(\bar{q})$ is diffeomorphic to the disjoint union of two 2-dimensional discs. Operating a blow up procedure by a factor $\Ga$ sufficiently big on  $(\Si'\cup\Si'_M)\cap B_{\de_0}(\bar{q})$ we obtain a surface $C^\infty$-close to the disjoint union of two 2-dimensional discs. By removing appropriate sets $\Ga D_1$ and $\Ga D_2$ diffeomorphic to a disc from each disconnected component, we see that we can connect the remaining surfaces (diffeomorphic to a disjoint union of two 2-dimensional annular surfaces) with a modification $\Ga \tC$ of the catenoid that is $C^2$ close to the standard catenoid and such that $\La(|D1\cup D_2|-|\tC|)|\le\ep$ and $\WW(\tC)\le\ep$ (for an explicit construction see \cite{Wo}). Rescaling back in $\Om$ and using the dilatation invariance we see that we have obtained a connected embedded surface $\Si_1\con\Om$ such that:
	\begin{displaymath}
	\begin{split}
	\WL(\Si_1) &=\WL(\Si')+\WL(\Si'_M) +\WW(\tC)-\La|\tC|-\WW(D1\cup D_2)+\La|D1\cup D_2| \\
	&\le  2\WL(\Si)+5\ep.
	\end{split}
	\end{displaymath}
	Now we can clearly iterate the procedure obtaining $\Si_2$, and in this case, by arbitrariness on the value of $\ep$, we can take a value $\ep^2$. Thus, using the notation above with an additional index 1 to distinguish from the previous quantities, we get a connected embedded surface $\Si_2\con\Om$ such that:
	\begin{displaymath}
	\begin{split}
	\WL(\Si_2)&=\WL(\Si_1)+\WL(\Si_{1,M_1})+\WW(\tC_1)+\La(|D_{1,1}\cup D_{2,1}|-|\tC_1|)-\WW(D_{1,1}\cup D_{2,1}) \\
	&\le 2(2\WL(\Si)+5\ep)+3\ep^2\\
	&=2^2 \WL(\Si)+5(2\ep)+3\ep^2.
	\end{split}
	\end{displaymath}
	So iterating the procedure taking $\ep^{n}$ when constructing $\Si_n$ we obtain:
	\begin{displaymath}
	\begin{split}
	\WL(\Si_n)&\le 2^n\WL(\Si)+5(2^{n-1}\ep)+3\sum_{i=2}^{n} 2^{n-i}\ep^i\\
	&\le 2^n\WL(\Si)+5(2^{n-1}\ep)+3\frac{2^{n-2}}{1-\ep} \quad\longrightarrow-\infty \qquad\qquad \mbox{as }n\to\infty,
	\end{split}
	\end{displaymath}
	being $\WL(\Si)<0$.
\end{proof}

\noindent The previous discussion allows us to solve completely Problem $(P)$ in the ball $B_1$:
\begin{thm}[Solution of $(P)_{B_1,\La}$)] \label{P2inB1}
	If $\Om=B_1$ the minimization problem $(P)_{B_1,\La}$ admits a solution if and only if $\La\le1$, in which case the minimum is $4\pi(1-\La)$. If $\La<1$ the unique minimizer is the unit sphere $S^2$. Moreover for all $\La>1$ the infimum of the problem is $-\infty$.
\end{thm}
\begin{proof}
	The last part of the statement is a consequence of Example \ref{exdante} and Lemma \ref{scattomenoinfinito}, in fact for $\La>1$ we have $\WL(S^1)<0$.\\
	If we consider $\La\le1$, applying Lemma \ref{thmwillmorevsareainequality}, we get $\WL(\Si)\ge |\Si|(1-\La)\ge 0$, hence as for the minimization problem we can restrict ourselves to connected surfaces. Moreover $\WL (\Si)\ge \WW(\Si)(1-\La)\ge4\pi(1-\La)=\WL(S^2)$. So $S^2$ is a minimizer.\\
	If $\La<1$, the uniqueness of the minimizer follows having $\WL(\Si)\ge\WW(\Si)(1-\La)\ge|\Si|(1-\La)$ for all $\Si$, with equality if and only if $\Si=S^2$.
\end{proof}

\begin{remark} [Upper Bound for $\LO$] \label{upperboundLO}
	Combining Example \ref{exdante} with the proof of Lemma \ref{scattomenoinfinito} we see that if there exist two balls $B_1(p),B_{1-\de}(p)$ such that $\bar{B}_1(p)\sm B_{1-\de}(p)\con\bar{\Om}$ for some $\de>0$ then a minimizing sequence of connected surfaces $(\Si_n)$ can be realized inside $\bar{B}_1(p)\sm B_{1-\de}(p)$ and thus for all $\La>1$ we have $C_\La=-\infty$. \\
	By rescaling invariance this means that if two balls $B_r(p),B_{r-\de}(p)$ are such that $\bar{B}_r(p)\sm B_{r-\de}(p)\con\bar{\Om}$ for a $\de>0$, then for all $\La>\frac{1}{r^2}$ we have $C_\La=-\infty$ (in other words $\LO\le \frac{1}{r^2}$).
\end{remark}

\vspace{0.5cm}
\noindent Now we turn our attention to the study of the general Problem $(P)_{\Om,\La}$. When no other is specified, $\Om$ is assumed to be open, with boundary of class $C^2$ and 
\[
\Om\con B_{\frac12}.
\]
By the rescaling invariance of Remark \ref{reminvarianza} we can do this without loosing any information. \\
Let us first make a simple observation.
\begin{remark}[Monotonicity]
	It is very important to keep in mind a simple monotonicity relation about functional $\WL$:
	\begin{equation}
	\La_1>\La_2 \qquad \Rightarrow \qquad \mathcal{W}_{\La_1} (\Si)<\mathcal{W}_{\La_2}(\Si) \quad\mbox{and}\quad C_{\La_1}\le C_{\La_2}.
	\end{equation}
\end{remark}

\noindent Moreover let us define a useful parameter:
\begin{equation}
\ep_\Om :=\sup \{ r>0: \exists\de>0,\exists B_r(p),B_{r-\de}(p)\mbox{ s.t. } \bar{B}_r(p)\sm B_{r-\de}(p)\con \bar{\Om}  \},
\end{equation}
so that by Remark \ref{upperboundLO} we have $\LO\le \frac{1}{\ep_\Om^2}$.
\begin{lemma} \label{lemmaref}
	If $\La>\LO$ then $C_\La=-\infty$.
\end{lemma}
\begin{proof}
	By definition of $\LO$ there is $\la\in(\LO,\La)$ such that $C_\la=-\infty$. By monotonicity for all $\Si\con\bar{\Om}$ we have:
	\begin{displaymath}
	\mathcal{W}_\la(\Si)> \mathcal{W}_\La(\Si),
	\end{displaymath}
	so $C_\La=-\infty$.
\end{proof}

\noindent Now we are able to give a first result about compactness.
\begin{thm}[Compactness for $\La<\LO$] \label{thmexistenceforLa<LO}
	If $\La<\LO$ and if $\Si_n$ is minimizing for $\WL$, then $\Si_n$ converges (up to subsequence) to a varifold $V\in\VV_2(\bar{\Om})$ with generalized mean curvature $H_V\in L^2_{(\bar{\Om},\mu_V)}$.
\end{thm}
\begin{proof}
	Let us take a minimizing sequence $(\Si_n)$ such that $\WL(\Si_n)\le C_\La+\frac{1}{n}$. Suppose that $|\Si_n|\to\infty$. Let $\la\in(\La,\LO)$, then:
	\begin{displaymath}
	0\le C_\la\le \mathcal{W}_\la(\Si_n)=\WL(\Si_n)-(\la-\La)|\Si_n| \to -\infty,
	\end{displaymath}
	that is impossible. So we have that there exists $L$ such that $|\Si_n|\le L$ for all $n$, and being $(\Si_n)$ a minimizing sequence then there also exists $C_0$ such that $\WW(\Si_n)\le C_0$ for all $n$. Moreover, denoting by $H_n$ the mean curvature vector of $\Si_n$ and calling again $\Si_n$ the varifold associated to $\Si_n$, for all $W\con \con \Om_\ep:=\ep$-neighborhood of $\Om$ we have (see Appendix A):
	\begin{displaymath}
	\begin{split}
	||\de  \Si_n||(W)&=\sup_{|X|\le1,\,\,\supp(X)\con W}\bigg| \de\Si_n(X)  \bigg|=\sup_{|X|\le1,\,\,\supp(X)\con W}\bigg| \int _{\Om_\ep}  \lgl X,H_n\rgl\, d\mu_{\Si_n} \bigg| \\
	&\le \sqrt{LC_0} \qquad\forall n.
	\end{split}
	\end{displaymath}
	So by compactness of varifolds (\cite{Al}, Appendix A) we get the existence of a limit $V\in\VV_2(\bar{\Om})$ of a subsequence $(\Si_{n_k})$ in the sense of varifolds. By lower semicontinuity we have that $V$ has mean curvature $H_V\in L^2_{(\bar{\Om},\mu_V)}$.
\end{proof}

\begin{remark}
	From the proof of Theorem \ref{thmexistenceforLa<LO} it is useful to remember that if a sequence $(\Si_n)$ of uniformly bounded surfaces has both uniformly bounded Willmore energy $\WW(\Si_n)$ and area $|\Si_n|$, then such sequence is precompact with respect to varifold convergence.
\end{remark}

\noindent Combining the information collected up to now we have the following consequence.
\begin{cor} The number $\LO$ is contained in the interval $[4,\frac{1}{\ep_\Om^2}]$.
\end{cor}
\begin{proof}
	First recall that we already observed that $\LO\le\frac{1}{\ep_\Om^2}$ (observe that $\ep_\Om\le\frac{1}{2}$ since $\Om\con B_{1/2}$). Now let us blow up $\Om$ by a factor 2; we get $2\Om\con B_1$ and $\WL(\Si)\ge|\Si|(1-\La)\ge0$ for all $\Si\con2\Om$ and for all $\La\le1$ by Lemma \ref{thmwillmorevsareainequality}. Thus we get $\La_{2\Om}\ge 1$, then rescaling back to $\Om$ we get $\LO\ge \frac{1}{(1/2)^2}=4$.
\end{proof}

\noindent Without further assumptions on $\Om$ we will see that we are not able to identify $\LO$ among its possible values (Example \ref{valoriLO}). We need some further results first.
\begin{remark}\label{remmasse}
	Let us consider two parameters $\La,(\La+\ep) \in (\La,\LO)$, and call $(\Si_n^\La)$ and $(\Si_n^{\La+\ep})$ two corresponding minimizing sequences. By the proof of Theorem \ref{thmexistenceforLa<LO} we know that areas $|\Si^\La_n|, |\Si^{\La+\ep}_n|$ are uniformly bounded; assume that there exist the limits of the sequences of their areas. It holds that:
	\begin{equation} \label{eqdecrescenzamasse}
	\lim_n|\Si_n^{\La}|\le \lim_n |\Si_n^{\La+\ep}|<+\infty,
	\end{equation}
	\begin{equation} \label{eqdecrescenzamasse2}
	\lim_n |\Si_n^{\La+\ep}|\le \lim_n \frac{\La-4}{\La} |\Si^\La_n|+\frac{1}{\La}\WW(\Si_n^{\La+\ep}),
	\end{equation}
	\begin{equation} \label{eqdecadimento}
	C_{\La+\ep}=\lim_n \mc{W}_{\La+\ep}(\Si_n^{\La+\ep})\le \lim_n \WW_\La(\Si_n^{\La})-\ep|\Si_n^{\La}|<C_\La.
	\end{equation}
	Let us prove such inequalities separately. First we have:
	\begin{displaymath}
	\begin{split}
	\lim_n\mathcal{W}_{\La+\ep}(\Si_n^{\La+\ep})&\le \lim_n \mathcal{W}_{\La+\ep}(\Si_n^{\La})=\lim_n\WL(\Si_n^{\La})-\ep |\Si_n^{\La}|+ \ep |\Si_n^{\La+\ep}|- \ep |\Si_n^{\La+\ep}|\\&\le\lim_n\WL(\Si_n^{\La+\ep})-\ep |\Si_n^{\La}|+ \ep |\Si_n^{\La+\ep}| - \ep |\Si_n^{\La+\ep}|\\&= \lim_n \mathcal{W}_{\La+\ep}(\Si_n^{\La+\ep})+\ep\lim_n |\Si_n^{\La+\ep}|-|\Si_n^{\La}|.
	\end{split}
	\end{displaymath}
	Since $|\Si^\La_n|, |\Si^{\La+\ep}_n|$ are uniformly bounded we get \eqref{eqdecrescenzamasse}. Moreover from $\lim_n \WW_\La(\Si^\La_n)\le \lim_n \WW_\La (\Si^{\La+\ep}_n)$ we get
	\begin{displaymath}
	\begin{split}
	\La\lim_n |\Si_n^{\La+\ep}|&\le \lim_n \WW(\Si_n^{\La+\ep})-\WW(\Si_n^{\La})+\La|\Si_n^{\La}|\le \lim_n \WW(\Si_n^{\La+\ep})+(\La-4)|\Si_n^{\La}|,
	\end{split}
	\end{displaymath}
	where in the second inequality we used Remark \ref{rmkwillmorevsarea2}. Hence we got \eqref{eqdecrescenzamasse2}.\\
	Finally
	\begin{displaymath}
	\lim_n\WW_{\La+\ep}(\Si_n^{\La+\ep})\le\lim_n  \WW_{\La+\ep}(\Si_n^{\La})=\lim_n \WW_\La(\Si_n^{\La})-\ep|\Si_n^{\La}|,
	\end{displaymath}
	that will give \eqref{eqdecadimento} once we prove that $|\Si_n^{\La}|\ge\de>0$ for all $n$. But in fact if $S$ is any sphere contained in $\Om$ we have that $\WL(S)<4\pi$ and thus $C_\La<4\pi$; therefore if by contradiction $\lim_n |\Si^\La_n|=0$ we would have $C_\La=\lim_n \WL(\Si^\La_n)\ge 4\pi$ that gives a contradiction.
\end{remark}

\noindent Now we are able to complete the characterization of the infima $C_\La$. recall that
\[
\LO:=\inf\{\La>0: C_\La =-\infty \}.
\]

\begin{thm}[Properties of $C_\La$] \label{thmbehaviour}
	The function $C_\La:\R_{>0}\to [-\infty,+\infty)$ that associates to the parameter $\La$ the corresponding infimum $C_\La$ has the following properties:
	\begin{enumerate} [label={\normalfont(\roman*)}]
	\item $C_0:=\lim_{\La\to0}C_\La = 4\pi$ independently of $\Om$, and $C_\LO\ge 0$,
	\item for $\La\in(0,\LO]$ the function $C_\La$ is continuous, nonnegative, concave and strictly decreasing. Moreover for all $\ep<\LO$ there exists $\de=\de(\ep)<0$ such that the derivative $C_\La'\le\de<0$ for almost all $\La\in(\LO-\ep,\LO]$ (i.e. where it exists), and $\de(\ep)$ cannot decrease as $\ep$ decreases,
	\item $C_\La=-\infty$ for each $\La>\LO$.
	\end{enumerate}
\end{thm}
\begin{proof}
	For any $\Si$, since $|\Si|\le\WW(\Si)$ and $\WW(\Si)\ge4\pi$, we have:
	\begin{displaymath}
	\WL(\Si)\ge \WW(\Si)(1-\La)\ge  4\pi(1-\La) \longrightarrow 4\pi \qquad \La\to0,
	\end{displaymath}
	then $C_0\ge4\pi$. Now take $r$ sufficiently small such that $S^2_r\con\bar{\Om}$, then:
	\begin{displaymath}
	C_\La\le\WL(S^2_r)= 4\pi-4\pi r^2 \La\longrightarrow 4\pi \qquad \La\to0,
	\end{displaymath}
	thus $C_0\le 4\pi$, and we got (i). \\
	For $\La\in(0,\LO)$ we already know from Equation \eqref{eqdecadimento} that in this interval the function is positive and strictly decreasing, thus it is differentiable almost everywhere and it has at most a finite number of jump-type discontinuities.\\
	Now we have:
	\begin{displaymath}
	\forall\Si \qquad \mathcal{W}_{\LO}(\Si)=\lim_{\La\to\LO^-} \WL(\Si),
	\end{displaymath}
	thus:
	\begin{displaymath}
	C_{\LO}=\inf_{\Si\con\bar{\Om}}\mathcal{W}_{\LO}(\Si)=\inf_{\Si\con\bar{\Om}} \lim_{\La\to\LO^-} \WL(\Si) \ge \inf_{\Si\con\bar{\Om}} \lim_{\La\to\LO^-} C_\La \ge 0,
	\end{displaymath}
	since $C_\La\ge0$ for all $\La<\LO$.\\
	Now we can prove continuity from the left. Take $\La_0\in(0,\LO)$ and suppose by contradiction that there exists $\eta>0$ such that $\lim_{\La\to\La_0^-}C_\La\ge C_{\La_0}+\eta$. Calling $(\Si_n)$ a minimizing sequence for the functional $\mathcal{W}_{\La_0}$, we have:
	\begin{displaymath}
	\begin{split}
	C_{\La_0}=\lim_{n}\mathcal{W}_{\La_0}(\Si_n)=\lim_{n}\lim_{\La\to\La_0^-}\WL(\Si_n)\ge \lim_{\La\to\La_0^-} C_\La \ge C_{\La_0}+\eta,
	\end{split}
	\end{displaymath}
	that is impossible.\\
	We can also prove continuity from the right. Take $\La_0\in(0,\LO)$ and suppose by contradiction that there exists $\eta>0$ such that $\lim_{\La\to\La_0^+}C_\La\le C_{\La_0}-\eta$. Call $(\Si_n^\La)$ a minimizing sequence for the functional $\mathcal{W}_{\La}$. If $\La_1\in(\La_0,\La_\Om)$, using Equation \eqref{eqdecrescenzamasse} we know that $\lim_n |\Si^\La_n| \le \lim_n |\Si^{\La_1}_n|\le L_1$ for any $\La\in(\La_0,\La_1)$. Therefore
	\begin{displaymath}
	\begin{split}
	C_{\La_0}&\ge \eta +\lim_{\La\to\La_0^+}\lim_n  \WL(\Si_n)= \eta +\lim_{\La\to\La_0^+}\lim_n \mathcal{W}_{\La_0}(\Si^\La_n) + (\La_0-\La)|\Si^\La_n|  \\
	&\ge \eta +  C_{\La_0} +\lim_{\La\to\La_0^+}(\La_0-\La)\lim_n |\Si_n^\La|  \\ 
	&\ge  \eta +  C_{\La_0} +\lim_{\La\to\La_0^+}(\La_0-\La)\lim_n |\Si_n^{\La_1}|\\
	&\ge  \eta +  C_{\La_0} +L_1\lim_{\La\to\La_0^+}(\La_0-\La)= \eta + C_{\La_0},
	\end{split}
	\end{displaymath}
	for some constant $L_1\ge \lim_n |\Si_n^{\La_1}| $, but that is impossible.\\
	We can also check continuity from the left in $\LO$. In fact let us consider a sequence $\La_n\to\LO^-$, then the functions $\WW_{\La_n}:\V:=\{V\in\VV_2(\bar{\Om}) : \exists H_V\in L^2_{(\bar{\Om},\mu_V)}  \}\to \R$ converge uniformly to the function $\WW_{\LO}:\V\to\R$ with respect to the $\bF$-metric of $\VV_2(\bar{\Om})$ on bounded sets (i.e. bounded in mass), that is:
	\begin{displaymath}
	\sup_{V\in\V,\,\,\bM(V)\le K}|\WW_{\La_n}(V)-\WW_{\LO}(V)|=	\sup_{V\in\V,\,\,\bM(V)\le K}|(\La_n-\LO)\bM(V)|\longrightarrow 0 \qquad n\to\infty,
	\end{displaymath}
	for all $K>0$. Hence we can swap the limit with the infimum in the following relation.
	\begin{displaymath}
	\begin{split}
	\lim_{\La_n\to\LO^-} \inf_{\Si\con\bar{\Om},\,\,|\Si|\le K} \WW_{\La_n}(\Si)=\inf_{\Si\con\bar{\Om},\,\,|\Si|\le K} \lim_{\La_n\to\LO^-} \WW_{\La_n}(\Si)=\inf_{\Si\con\bar{\Om},\,\,|\Si|\le K} \WW_{\LO}(\Si),
	\end{split}
	\end{displaymath}
	for all $K\ge0$. Hence:
	\begin{equation}\label{A}
	\lim_{K\to\infty}\lim_{\La_n\to\LO^-} \inf_{\Si\con\bar{\Om},\,\,|\Si|\le K} \WW_{\La_n}(\Si)=C_\LO.
	\end{equation}
	If we are able to swap the first two limits in \eqref{A}, we are done. Let
	\[
	C_{K,n}:= \inf_{\Si\con\bar{\Om},\,\,|\Si|\le K} \WW_{\La_n}(\Si) \ge 0
	\]
	Since $C_{K,n}$ is decreasing in the two indexes and the numbers $C_{K,n}$ are greater than or equal of zero we have
	\[
	\inf_K C_{k,n} =\lim_{K\to\infty} C_{K,n} \le \lim_{K\to\infty} C_{K,m} = \inf_K C_{K,m} \qquad\forall\,n>m,
	\]
	and then
	\[
	\inf_n \inf_K = \lim_n \inf_K C_{K,n} = \lim_n \lim_K C_{K,n}.
	\]
	Similarly we get
	\[
	\inf_K \inf_n C_{K,n} = \lim_K \lim_n C_{K,n},
	\]
	and thus
	\[
	\lim_K \lim_n C_{K,n} = \lim_n \lim_K C_{K,n}= \inf_{K,n} C_{K,n}.
	\]
	Using \eqref{A} we conclude that
	\begin{displaymath}
	C_\LO=\lim_{K\to\infty}\lim_{\La_n\to\LO^-} C_{K,n}=\lim_{\La_n\to\LO^-}\lim_{K\to\infty}C_{K,n}=\lim_{\La_n\to\LO^-} C_\La.
	\end{displaymath}
	Now using Equation \eqref{eqdecadimento} and reminding that $C_\La$ is differentiable for almost all $\La$, we see that for almost all $\La$ the function $f_\La:\ep\mapsto C_{\La+\ep}$ is such that $f_\La'(0)\le-\lim_n|\Si^\La_n|<0$ for almost all $\ep$ for which $f_\La$ is defined. Using now Equation \eqref{eqdecrescenzamasse}, we get $f'_{\La+\ep}(0)\le f'_\La(0)<0$ for almost all $\ep$ for which the relation is defined. Again by Equation \eqref{eqdecrescenzamasse} we see that $\de$ cannot decrease as $\ep$ decreases, thus we have completed the proof of (ii). We already know that (iii) is true by Lemma \ref{lemmaref}, thus we completed the proof of the theorem.
\end{proof}

\noindent For the convenience of the reader let us recall the definitions:

\begin{equation}\label{definitions}
\W(\Si):=\frac{\WW(\Si)}{|\Si|},\qquad \tLO:=\inf\{ \W(\Si):\Si\con\bar{\Om} \}.
\end{equation}

\noindent We want to prove some results that motivate the connection between $\WW_\LO$ and $\W$.
	
	\begin{prop} \label{LO=tLO}
		It holds that
		\begin{equation}
		\LO=\tLO.
		\end{equation}
		Moreover
		\begin{enumerate} [label={\normalfont(\roman*)}]
			\item if a minimizing sequence $(\Si_n)$ for the functional $\W$ satisfies that $|\Si_n|\le L$ for any $n$, then it is also minimizing for the functional $\WW_\LO$ and $C_\LO=0$,
			\item if a minimizing sequence $(\Si_n)$ for the functional $\mathcal{W}_{\LO}$ satisfies that $|\Si_n|\le L$ for any $n$ and if $C_\LO=0$, then $\Sigma_n$ is also a minimizing for $\W$.
		\end{enumerate}
	\end{prop}
	\begin{proof}
		For all $\Si$ it holds:
		\begin{displaymath}
		\W(\Si)= \frac{\WW_{\LO}(\Si)}{|\Si|}+\LO\ge \frac{C_{\LO}}{|\Si|}+\LO\ge\LO,
		\end{displaymath}
		thus $\tLO\ge\LO$.\\
		Now let us take a sequence $\La_n\to \LO^+$ and surfaces $\Si_n$ such that $\mathcal{W}_{\La_n}(\Si_n)\le 0$ for all $n$. Then:
		\begin{displaymath}
		\W(\Si_n)\le\La_n\to\LO\le\tLO,
		\end{displaymath}
		thus $\LO=\tLO$. Now we prove the remaining two statements separately.\\
		(i) Since $|\Si_n|\le L$ for some constant $L$, and we have:
		\begin{displaymath} 
		0\le \WW_\LO(\Si_n)=|\Si_n|(\W(\Si_n)-\LO)\le L(\W(\Si_n)-\tLO)\longrightarrow0 \qquad n\to\infty.
		\end{displaymath}
		(ii) If $\Si_n$ is minimizing for $\WW_\LO$, then $|\Si_n|\ge\de>0$, otherwise $\WW_\LO(\Si_n)\to C_\LO\ge4\pi$, but $C_\LO=0$ by hypothesis. We have:
		\begin{displaymath}
		0\le \de(\W(\Si_n)-\tLO)\le |\Si_n|(\W(\Si_n)-\LO)=\WW_\LO(\Si_n)\longrightarrow0  \qquad n\to\infty.
		\end{displaymath}
		Hence $\W(\Si_n)\to\tLO$.
	\end{proof}

\noindent Let us say that a functional $F:\{\Si\con\bar{\Om}\}\to\R$ is coercive if there exists $L>0$ such that $\inf F(\Si) = \inf_{|\Si|\le L} F(\Si)$. With this definition we see that Proposition \ref{LO=tLO} shows that the fact that $C_\LO=0$ is strictly related to the coerciveness of $\W$. More precisely we can state the following.

\begin{cor}\label{cor:coercive}
	If $\WW_\LO$ is coercive and $C_\LO=0$, then $\W$ is coercive. In particular, if $\Om=B_1$ the functional $\W$ is coercive.
\end{cor}

\begin{proof}
	The proof immediately follows by item ii) in Proposition \ref{LO=tLO} and by Theorem \ref{P2inB1}.
\end{proof}

\noindent In the following examples we show that it is not possible in general to identify the value of $\LO$ in the interval $[4,1/\ep_\Om^2]$.
\begin{example} \label{valoriLO}Let us illustrate three examples.
	\begin{enumerate} [label={\normalfont(\roman*)}]
	\item A first simple example is $\Om=B_{1/2}\sm B_{1/2-\ep}$ with $0<\ep<1/2$, for which we have $4=\LO=\frac{1}{\ep_\Om^2}$.
	\item Now we construct an example in which $4<\LO=\frac{1}{\ep_\Om^2}$. Let us consider $1/4<r<1/2$, $\de=1/2-r<r$ and let $\Om=B_r\cup B_{\de/4}\big(\big(r+\frac{3}{4}\de,0,0\big)\big)$. In this case ${\rm diam}(\Om)=1$ but clearly $\ep_\Om=r<1/2$ and $\LO=\frac{1}{r^2}=\frac{1}{\ep_\Om^2}>4$.
	\item We can also construct an example in which $\LO<\frac{1}{\ep_\Om^2}$. Let us denote by $E_{a,c}=\{(x,y,z)\in\R^3: x^2/a^2+ y^2/a^2+ z^2/c^2=1\}$. Let us fix $c=1/2$ and $c-\eta<a<c$ for $\eta>0$ sufficiently small such that $\WW(E_{a,1/2})\le 10$ (this is possible since when $a=1/2$ we would obtain a sphere).\\
	Now we consider $\Om$ as the volume enclosed by $E_{a,1/2}$ except the volume enclosed by $\{p-\de\nu(p):p\in E_{a,1/2} \}$ with $\nu$ outer normal of $E_{a,1/2}$ and $\de<<1$. For $\de$ sufficiently small we get that $\ep_\Om=\de/2$ and:
	\begin{displaymath}
	\WL(E_{a,1/2})\le 10-\La|E_{a,1/2}|< 10-\La 4\pi a^{4/3}\frac{1}{2^{2/3}},
	\end{displaymath}
	where we used the isoperimetric inequality $(4\pi)^{1/3}3^{2/3} |A|^{2/3} \le |\pa A|$ for $A\con\R^3$. Finally we observe that for $\de$ sufficiently small there exists $\La<1/\de^2$ such that $\WL(E_{a,1/2})=0$. This implies $\LO< 1/\de^2=1/\ep_\Om^2$ as desired.
	\end{enumerate}
\end{example}

\noindent The following examples point out the strong dependence of the problems on the geometry of the domain as it is taken unbounded. The scenery seems to become somehow chaotic, in the sense that we did not find spontaneous hypotheses on an unbounded $\Om$ under which general conclusions can be derived.
\begin{example} \label{Om1P}
	Let us take:
	\begin{displaymath}
	\Om=\{ (x,y,z)\in\R^3| x^2+y^2< 1 \}.
	\end{displaymath}
	Let us consider a sequence of surfaces $\Si_n$ that is $C^1$-close to $\Si_n=S_1\cup C_n \cup S_2$ where $C_n$ is a cylinder or radius 1 and height $n$, while $S_1$ and $S_2$ are the two hemispheres of the standard $S^2$ translated in a way in which $\Si_n$ becomes an admissible surface. We can arrange:
	\begin{displaymath}
	\WL(\Si_n)\le4\pi(1-\La)+\bigg( \frac{1}{4}  -\La  \bigg)2\pi n+\de,
	\end{displaymath}
	for some $\de>0$, where the first term is the energy of the two hemispheres and the second one is due to the cylinder. Then $\WL(\Si)$ converges to $-\infty$ as $n$ increases if $\La>1/4$, so $\La_{\Om}\le 1/4$.\\
	Being $\Om$ unbounded we cannot use the results obtained above, and it is also interesting to notice that the direct method consisting of taking a minimizing sequence and proving its convergence in the sense of varifolds is no longer applicable, since in this case we apparently have no tools in order to uniformly estimate the area of the sequence.
\end{example}

\noindent Considering different unbounded domains $\Om$ the situation may degenerate, as shown in the next example.
\begin{example} \label{Om2P}
	Let us take:
	\begin{displaymath}
	\Om=\{ (x,y,z)\in\R^3: |z|<1 \}.
	\end{displaymath}
	In this case we will see that the problems become immediately trivial. Let us consider the sequence of surfaces $\Si_n$ $C^1$-close to $\si_n^1\cup \si_n^2\cup T_n$, where $\si_n^i$ are two discs of radius $n$ with center $(0,0,-1)$ or $(0,0,1)$ lying on the opposite sides of $\pa \Om$, and $T_n$ is the subset with positive Gaussian curvature of the torus given by the rotation of a circumference of radius $1$ at a distance $n$ from the axis $z$. We can estimate for some $\de>0$ that:
	\begin{displaymath}
	\WL(\Si_n)\le Cn-2\pi\La n^2+\de\longrightarrow-\infty \qquad n\to\infty,
	\end{displaymath}
	for all $\La>0$. So there is not a minimum for $\WL$ and the infimum of the problem is $-\infty$.
\end{example}

\noindent Let us conclude with a further example.

\begin{example}
	It is not true in general that the boundary $\pa \Om$ of a bounded convex domain is a minimizer for Problem $(P)_{\Om,\La}$ for all $\La\le\LO$. Take for example $\pa\Om$ $C^1$-close to $S_1\cup C_h\cup S_2$ where $S_1$ and $S_2$ are translations of the two hemispheres of the standard $S^2$ and $C_h$ is a cylinder of radius one and height $h$. $\Om$ is the bounded set with such boundary. We can arrange that:
	\begin{displaymath}
	\WW_\La(\pa \Om)\ge\WW_\La(S^2)+\bigg( \frac{1}{4} -\La \bigg)2\pi h-\de > \WW_\La(S^2),
	\end{displaymath}
	for some $\de>0$ for each $\La<\frac{1}{4}-\frac{\de}{2\pi h}$. Being $S^2\con\bar{\Om}$ we see that the boundary cannot be a minimizer for $\WW_\La$ for all $\La\le\LO$.
\end{example}

\section{Regularity}

In this section we prove statement $(iii)$ of Theorem A. We adopt the convention that if $L$ is a plane in $\R^3$, then we write $u=(u_1,u_2,u_3)\in C^r(\bar{A};L^\perp)$, where $A\con L$, if $u(x)\in L^\perp\,\forall x\in A$. In this case we write:
\begin{equation} \label{graph}
\mbox{{\rm graph} } u =\{ x+u(x)|x\in A \}.
\end{equation}
Let us first recall the two main tools that we will use in the proof.
\begin{lemma}[Graphical Decomposition, \cite{SiEX}]\label{graphdecomp}
	Let $\Si$ be a compact surface without boundary with $0\in \Si$. Then for any $\be>0$ there exists $\ep_0$ (independent of $\Si,\ro$) such that if $\ep\in(0,\ep_0]$, $|\Si\cap \bar{B}_\ro |\le \be \ro^2$ and $\int_{\Si\cap B_\ro} |A|\le\ep\ro$, then the following holds.\\
	There are disjoint closed sets $P_1,...,P_N\con\Si$ such that:
	\begin{displaymath}
	\sum_{j=1}^{N}\mbox{{\rm diam} }P_j\le C\ep^{1/2}\ro
	\end{displaymath}
	and
	\begin{displaymath}
	\Si\cap B_{\ro/2}\sm \bigg(\bigcup_{j=1}^N P_j\bigg)=\bigg(\bigcup_{i=1}^M \mbox{{\rm graph} }u_i\bigg)\cap B_{\ro/2},
	\end{displaymath}
	where $u_i\in C^\infty (\bar{A}_i;L_i^\perp)$, with $L_i$ plane, $A_i$ smooth bounded connected open in $L_i$ of the form $A_i=A_i^0\sm (\cup_k d_{i,k})$ where $A_i^0$ is simply connected and $d_{i,k}$ are closed disjoint discs in $L_i$ not intersecting $\pa A_i^0$ and also $\sum_{i,k} {\rm diam}(d_{i,k})\le C\ep^{1/2}\ro$, $\sum_{i,k} |d_{i,k}|\le C\ep \rho^2$.\\
	Moreover $\mbox{{\rm graph} }u_i$ is connected and:
	\begin{displaymath}
	\sup_{A_i} \frac{|u_i|}{\ro}+\sup_{A_i} |Du_i|\le C\ep^{1/6}.
	\end{displaymath}
	If we also have $\int_{B_\ro}|A|^2\le \ep^2$, then in addition to the above conclusions it holds that for every $\si\in(\ro/4,\ro/2)$ such that $\pa B_\si$ intersects $\Si$ transversely and $\pa B_\si \cap (\cup_j P_j)=\emptyset$, we have:
	\begin{displaymath}
	\Si\cap \bar{B}_\si=\bigcup_{i=1}^M D_{\si,i},
	\end{displaymath}
	where each $D_{\si,i}$ is homeomorphic to a disc and $\mbox{{\rm graph} } u_i\cap \bar{B}_\si \con D_{\si,i}$. Also $ D_{\si,i}\sm \mbox{{\rm graph} } u_i$ is a union of a subcollection of the $P_j$ and each $P_j$ is homeomorphic to a disc.
\end{lemma}

\begin{lemma}[Comparison, \cite{Schy}] \label{comparison}
	Let $L$ be a plane, $x_0\in L$, $u\in C^\infty(U;L^\perp)$ where $U\con L$ is an open neighborhood of $L\cap \pa B_\ro(x_0)$ and assume $|Du|\le C$ on $U$. Then there exists a function $w\in C^\infty(\overline{B_\ro(x_0)}\cap L;L^\perp)$ such that
	\begin{equation}
	\begin{split}
	&w=u, \qquad \pa_\nu w= \pa_\nu u \qquad\qquad\mbox{on }\pa B_\ro(x_0)\cap L,\\
	& \frac{\|w\|_{L^\infty(B_\ro(x_0)\cap L)}}{\rho}\le c(n) \bigg( \frac{\|u\|_{L^\infty(\pa B_\ro(x_0)\cap L)}}{\rho}
	+ \|Du\|_{L^\infty(\pa B_\ro(x_0)\cap L)} \bigg),\\
	&\|Dw\|_{L^\infty(B_\ro(x_0)\cap L)}\le c(n)\|Du\|_{L^\infty(\pa B_\ro(x_0)\cap L)},\\
	&\int_{(B_\ro(x_0)\cap L)} |D^2w|^2 \le c(n)\ro \int_{\mbox{{\rm graph} }(u|_{L\cap \pa B_\ro(\xi)})} |A|^2\,d\HH^1,
	\end{split}
	\end{equation}
	where $\pa_\nu$ denotes the normal outward derivative and $A$ is the second fundamental form of $\mbox{{\rm graph} }(u)$.
\end{lemma}


\noindent Now we can prove the regularity result. We will make use of arguments in \cite{SiEX}, so we will mainly focus on the differences that arise in our problem, namely the additional area term and the confinement in $\Om$. The feeling is that this method is very well applicable for functionals given by the sum of the Willmore energy and some lower order term.

\begin{thm}[Regularity]\label{thmregularity}
	If $\La>0$ is sufficiently small (depending on $\Om$), then there exists an embedded surface $\Si\con\bar{\Om}$ of class $C^{1,\al}\cap W^{2,2}$ such that $\WL(\Si)=C_\La$. Moreover the surface $\Si\cap \Om$ is of class $C^\infty$.
\end{thm}

\noindent Let us briefly illustrate the strategy of the proof of Theorem \ref{thmregularity}. We are going to consider a minimizing sequence $(\Si_n)$ for the functional $\WL$ with $\La<\LO$ converging to some varifold $V$. For $\La$ small enough we will have that $\WW(\Si_n)\le8\pi-\de$ and $V$ has multiplicity $1$. The analysis of the regularity of the support of $V$ is divided into two steps. We can distinguish finitely many points $\xi_1,...,\xi_P \in \supp V$, called bad points, that are points that can accumulate energy in the limit. First we will study the regularity at points $p\in \supp V \sm \{\xi_1,...,\xi_P\}$, in fact around such points we will be able to apply the Graphical Decomposition Lemma \ref{graphdecomp}. The graphical decomposition will be applied to any $\Si_n$ of the minimizing sequence around the same chosen point $p\in \supp V \sm \{\xi_1,...,\xi_P\}$; in such a way we will be able to replace controlled pieces of $\Si_n$ with comparison graphs given by Lemma \ref{comparison}. The minimizing property of the sequence $(\Si_n)$ thus yields inequalities by comparing $(\Si_n)$ with the modified sequence. This will lead to the decay inequality \eqref{stimapsi}, that readily implies $C^{1,\alpha}$ regularity around the good point $p$. Then a bootstrap argument based on the elliptic equation satisfied by critical points gives $C^\infty$ regularity of $\supp V\sm \{\xi_1,...,\xi_P\}$ inside $\Om$ and $C^{1,\alpha}\cap W^{2,2}$ of $\supp V\sm \{\xi_1,...,\xi_P\}$ in $\bar{\Om}$. The study of the regularity around a chosen bad point $\xi$ is similar in the spirit, but more careful. By uniform bounds one can identify around $\xi$ a ball $B_\tau(\xi)$ such that the convergence $\Si_n\to V$ is smooth outside such ball; controlling the oscillation of the tangent planes in suitable annular regions around $\Si_n\cap \pa B_\tau(\xi)$, we will replace part of $\Si_n\cap B_\tau(\xi)$ with suitable controlled disks, and we will argue again by comparison with the original minimizing sequence. This yields estimates completely analogous to the ones of the first case, and one derives regularity around the bad point as well.

\begin{proof}[Proof of Theorem \ref{thmregularity}]
	For $\La<\LO$, let us consider a minimizing sequence $(\Si_n)$ for $\WL$ converging in the sense of varifolds to $V\in\VV(\bar{\Om})$. For a given $\ep>0$, we say that a point $\xi\in\R^3$ is a \emph{bad point} if
	\begin{equation}
	\lim_{\ro\searrow0}\bigg( \liminf_{n\to+\infty}\int_{\Si_n\cap B_\ro(\xi)}|A_n|^2 \bigg)>\ep^2,
	\end{equation}
	where $A_n$ is the second fundamental form of $\Si_n$ and $|A_n|$ is its norm. If a point $\xi\in\R^3$ is not a bad point we then call it a \emph{good point}. Now we show that there is only a finite number of bad points.\\
	Fix  some $\bar{\La}<\LO$. We know from Remark \ref{remmasse} that there exists sequences $(\Si_n^\La)$ and $(\Si_n^{\bar{\La}})$ that are minimizing respectively for the parameters $\La$ and $\bar{\La}$ and they converge in the sense of varifolds, and then:
	\begin{displaymath}
	\lim_n |\Si_n^\La|\le\lim_n |\Si_n^{\bar{\La}}|=:m(\bar{\La}).
	\end{displaymath}
	In the following we choose $\La<\min\left\{\LO,\frac{4\pi}{m(\bar{\La})}\right\}$, so that $\La\lim_n |\Si_n^\La|<4\pi$. Hence, since $C_\La<4\pi$, we get that $\WW(\Si_n^\La)\le8\pi-\de$ for $n$ big enough and some $\de>0$. This implies that $\Si_n^\La$ is embedded by Theorem \ref{thmlowerbound} for $n$ big enough and that the genus of $\Si_n$ is bounded, in fact the minimum Willmore energy at genus $g$ is less then $8\pi$ and converges to $8\pi$ as $g\to\infty$ (see \cite{KuLiSc}).\\
	The above discussion has also another consequence: let $g_n$ be the genus of $\Si_n$, then $g_n\in\{0,1,...,\bar{g} \}$ for some $\bar{g}\in\N$ big enough. Hence there is a convergent subsequence $g_{n_j}$. This means that $g_{n_j}$ is constant for $j$ big enough. Then replacing $\Si_n$ with $\Si_{n_j}$ we get a minimizing sequence that has definitely constant genus. Hence we can assume without loss of generality that $\Si_n$ has fixed genus $g$ for all $n$. \\
	Another consequence is that, since by lower semicontinuity we have $\WW(V)<8\pi$, then $V$ has multiplicity $1$ $\mu_V$-almost everywhere by Remark \ref{remmoltvar}. \\
	We can apply Gauss-Bonnet Theorem to get:
	\begin{displaymath}
	\frac{1}{4}\int_ {\Si_n} |A_n|^2 =\WW(\Si_n)-\frac\pi2 (2-2g),
	\end{displaymath}
	with $g$ the genus of $\Si_n$ (the same for all $n$). Being $\Si_n$ minimizing, we have that $\int_ {\Si_n} |A_n|^2$ is bounded. So if $N$ is the number of bad points related to $\ep>0$, we get:
	\begin{displaymath}
	N\ep^2\le \liminf_n \int_ {\Si_n} |A_n|^2,
	\end{displaymath}
	giving an upper bound on $N$ in term of $\ep$.\\
	Moreover, by modifying the minimizing sequence with small perturbations, we can assume that $\Si_n\con\Om$ for every $n$ without loss of generality.\\
	The monotonicity formula (\cite{SiEX} Equation (1.3), or Appendix B in \cite{NoPo19}) implies that
	\begin{equation}\label{eq:StimaDensita}
	\frac{|\Si_n\cap B_\ro(p)|}{\ro^2}\le \frac{3}{2}\WW(\Si_n),
	\end{equation}
	for any $\ro>0$ and any $p\in \Si_n$. Hence we can take $\be=12\pi>\frac{3\WW(\Si_n)}{2}$ in Lemma \ref{graphdecomp} and let $\ep_0$ be the corresponding number given by such lemma.\\
	Let us fix an arbitrary $\ep\in(0,\ep_0)$; from now on we will call $\xi_1,...,\xi_P$ the bad points related to such $\ep$.\\
	For any $\xi\in supp(V)\sm\{\xi_1,...,\xi_P \}$ we can select $\ro(\xi,\ep)>0$ such that for all $\ro'\le\ro(\xi,\ep)$ we have $\int_{\Si_n\cap B_{\ro'}(\xi)}|A_n|^2\le \ep^2$ for infinitely many $n$; hence the last part of Lemma \ref{graphdecomp} is applicable to $\Si_n$ in $B_{\ro'}(\xi)$ for infinitely many $n$. Also by \eqref{eq:StimaDensita} we get that there exists $r\in(0,\ro']$ such that $|\pa B_r(p)\cap \Si_n|\le 3\be r$. Taking $\ro=\min\{\ro',r\}$ we can Lemma \ref{teta} for $n$ large enough with $\te$ small enough fixed (independent of $n,\ep,\xi$). We deduce that only one of the discs $D_j^{(n)}$, for example $D_1^{(n)}$, given by Lemma \ref{graphdecomp} can intersect the ball $B_{\te\ro}(\xi)$. Also, for infinitely many $n$ we know that there exist a plane $L_n$ containing $\xi$ and a $C^\infty(\bar{\Omega}_n)$ function $u_n:\bar{\Omega}_n\to L_n^\perp$ such that:
	\begin{equation} \label{ref}
	\begin{split}
	&\frac{|u_n|}{\ro}+|Du_n|\le C\ep^{1/6}, \\
	(\mbox{{\rm graph} }u_n &\cup_j P_{n,j})\cap B_\si(\xi)=D_1^{(n)}\cap B_\si(\xi),\\
	&\sum_j {\rm diam}(P_{n,j})\le C\ep^{1/2}\ro,
	\end{split}
	\end{equation}
	where each $P_{n,j}$ is diffeomorphic to a closed disc disjoint from $\mbox{{\rm graph} }(u_n|_{\Om_n})$ and $\si\in (\te\ro/2,\te\ro)$ is independent of $n$.\\	
	Now let us consider $C_\si(\xi):=\{ x+y|x\in B_\si(\xi)\cap L_n,y\in L_n^\perp \}$; by the Selection Principle \ref{selectionprinciple} there exists a set $T\con (\te\ro/2,\te\ro)$ of measure $\ge \te\ro/8$ such that for each $\si\in T$ we have $\pa C_\si(\xi)\cap P_{n,j}=\emptyset$ for infinitely many $n$. Hence for infinitely many $n$ we can apply Lemma \ref{comparison} on $D_1^{n}\cap B_\si(\xi)$ to get a function $w_n$ on $B_\si(\xi)\cap L_n$ such that:
	\begin{displaymath}
	\int_{L_n\cap B_\si(\xi)} |D^2w_n|^2\le C\si\int_{\Ga_n}|A_n|^2\,d\HH^1, 
	\end{displaymath}
	with $\Ga_n=\mbox{{\rm graph} }(w_n|_{L_n\cap \pa B_\si(\xi)})$ (the integration on subsets of planes $L_n$ is always understood with respect to the Lebesgue measure on such planes).\\
	Let $\tA_n$ be the second fundamental form of $\mbox{{\rm graph} }w_n$, in particular we have:
	\begin{displaymath}
	\int_{{\rm graph}(w_n)} |\tA_n|^2\le C\si\int_{\Ga_n}|A_n|^2\,d\HH^1.
	\end{displaymath}
	Note that by the estimates in Lemma \ref{comparison}, by choosing $\si$ sufficiently small (depending on $\xi$), we can assume $\mbox{{\rm graph} }w_n\con\bar{\Om}$. Then the $C^{1,1}$ surface $\tS_n:=(\Si_n\sm( D_1^{(n)}\cap B_\si(\xi)))\cup \mbox{{\rm graph} }w_n$ is such that $\WL(\Si_n)\le\WL(\tS_n)+\ep_n$ for some $\ep_n\searrow0$. Now we argue like in \cite{SiEX}, except that here we have to control the area term in the energy.\\
	Since $\Si_n$ has the same genus of $\tS_n$, by the Gauss-Bonnet Theorem we also get:
	\begin{displaymath}
	\ep_n+\int_ {\tS_n} \bigg(\frac{1}{4}|\tA_n|^2-\La\bigg) \ge+ \int_ {\Si_n} \bigg(\frac{1}{4}|A_n|^2-\La\bigg).
	\end{displaymath}
	So we have that:
	\begin{displaymath}
	\int_ {D_1^{(n)}\cap B_\si(\xi)} \frac{1}{4}|A_n|^2\le \ep_n+ \int_{ {\rm graph} (w_n)} \frac{1}{4}|\tA_n|^2 \,d\HH^2 +\La\bigg( \int_{D_1^{(n)}\cap B_\si(\xi)} d\HH^2-\int_{ {\rm graph} (w_n)} d\HH^2 \bigg) .
	\end{displaymath}
	Using Equations \eqref{ref}, let us estimate:
	\begin{displaymath}
	\begin{split}
	\int_{D_1^{(n)}\cap B_\si(\xi)} d\HH^2 &\le \int_ {\pi_{L_n}({\rm graph}(u_n)\cap B_\si(\xi)} \sqrt{1+|D u_n|^2}\,d\LL^2 +\sum_j |P_{n,j}|\\&\le \sqrt{1+c\ep^{1/3}}|\pi_{L_n}({\rm graph}(u_n)\cap B_\si(\xi)| +C\ro^2\le (\sqrt{1+c\ep^{1/3}}\pi\te^2+C)\ro^2=:\frac{1}{4}a\ro^2.
	\end{split}
	\end{displaymath}
	Hence:
	\begin{displaymath}
	\int_ {D_1^{(n)}\cap B_\si(\xi)} |A_n|^2\le 4\ep_n+a\ro^2+\int_{ {\rm graph} (w_n)} |\tA_n|^2 \le 4\ep_n+a\ro^2+C\si \int_{\Ga_n}|A_n|^2\,d\HH^1.
	\end{displaymath}
	That is:
	\begin{equation}\label{ref1}
	\int_{ \Si_n\cap B_\si(\xi)}|A_n|^2\le 4\ep_n+a\ro^2+C\si\int_{ \pa( D_1^{(n)}\cap B_\si(\xi)) } |A_n|^2\,d\HH^1.
	\end{equation}
	Since $\si$ was selected arbitrarily from the set $T$ of measure $\ge \te\ro/8$ in the interval $(\te\ro/2,\te\ro)$ we can arrange that:
	\begin{displaymath}
	\int_{ \pa( D_1^{(n)}\cap B_\si(\xi)) } |A_n|^2\,d\HH^1\le \frac{4}{\si}\int_ {\Si_n\cap B_{\te\ro}(\xi)\sm B_{\frac{\te\ro}{2}}(\xi)} |A_n|^2
	\end{displaymath}
	for infinitely many $n$. So using Equation \eqref{ref1}, for all $\ro\le\te\ro(\ep,\xi)$ we get:
	\begin{displaymath}
	\int_{ \Si_n\cap B_{\frac{\ro}{2}}(\xi)}|A_n|^2\le 4\ep_n +a\ro^2+C\int_ {   \Si_n\cap B_{\te\ro}(\xi)\sm B_{\frac{\te\ro}{2}}(\xi)       }  |A_n|^2.
	\end{displaymath}
	Adding $C$ times the left side we obtain:
	\begin{displaymath}
	\int_{ \Si_n\cap B_{\frac{\ro}{2}}(\xi)}|A_n|^2\le \ep_n+\al\ro^2+\ga \int_{ \Si_n\cap B_{\ro}(\xi)}|A_n|^2,
	\end{displaymath}
	where $\ga=\frac{C}{C+1}\in(0,1)$ and we named $\ep_n$ and $\al$ respectively the quantities $\frac{4}{C+1}\ep_n$ and $\frac{a}{C+1}$.\\
	Defining:
	\begin{equation}
	\psi(\ro,\xi):=\liminf_n \int_{\Si_n\cap B_\ro(\xi)}|A_n|^2,
	\end{equation}
	we get the following decay relation:
	\begin{equation} \label{decad}
	\psi\bigg(\frac{\ro}{2},\xi\bigg)\le \ga \psi(\ro,\xi) +\al\ro^2.
	\end{equation}
	Now let us observe that if $\xi_0\in supp(V)\sm\{\xi_1,...,\xi_P \}$, we can take:
	\begin{displaymath}
	\ro(\xi,\ep)=\frac{\ro(\xi_0,\ep)}{2}
	\end{displaymath} 
	for all $\xi\in supp(V)\cap B_{\frac{\ro(\xi_0,\ep)}{2}}(\xi_0)$. Hence, fixed $\xi_0\in supp(V)\sm\{\xi_1,...,\xi_P \}$, Equation \eqref{decad} holds for all $\xi\in supp(V)\cap B_{\frac{\ro(\xi_0,\ep)}{2}}(\xi_0)$ and for all $\ro\le\te\ro(\xi_0,\ep)/2:=\ro_0$. The constant $C$ defining $\ga$ is the one given by Lemma \ref{comparison}, so we can choose it arbitrarily big in order to get $\ga=\frac{C}{C+1}\in (1/2,1)$ and $\al=\frac{a}{C+1}\in(0,1/8)$. Hence given $\xi_0\in supp(V)\sm\{\xi_1,...,\xi_P \}$ we can apply Corollary \ref{cordecad} to get:
	\begin{equation} \label{stimapsi}
	\begin{split}
	&\psi(\ro,\xi)\le C\bigg(\frac{\ro}{\ro_0} \bigg)^\be \psi(\ro_0,\xi)\le C\bigg(\frac{\ro}{\ro_0} \bigg)^\be \psi(\ro(\xi_0,\ep),\xi_0) \\& \forall \xi\in supp(V)\cap B_{\frac{\ro(\xi_0,\ep)}{2}}(\xi_0),\,\,\forall \ro\le\ro_0:=\te\ro(\xi_0,\ep)/2,
	\end{split}
	\end{equation}
	for some $C>0,\be\in(0,1)$, where second inequality holds since $\psi(\ro_0,\xi)\le \psi(\ro(\xi_0,\ep),\xi_0) $. \\
	Hence we ultimately got the key decay relation on the second fundamental form (the same of Equation (3.2) in \cite{SiEX}).
	So following the same arguments in \cite{SiEX} (page 301) one gets that the varifold $V$ has a multiplicity $1$ tangent plane at each point $\xi\in supp(V)\cap B_\ro(\xi_0)$ with a normal vector $N(\xi)$ such that $\|N(\xi_1)-N(\xi_2)\|\le C|\xi_1-\xi_2|^\al$ for all admissible $\xi_1,\xi_2$. Also this means that if $U$ is a sufficiently small neighborhood of $\xi_0$, we have:
	\begin{displaymath}
	\mu_V\res U =\HH^2\res (\Si\cap U),
	\end{displaymath}
	where $\Si$ is a $C^{1,\al}$ surface.
	%
	Moreover from \eqref{stimapsi} one gets $\int_{ \Si\cap B_\ro(\xi)} H^2\le C\ro^\al$ for $\xi$ that is not a bad point, and this decay implies that $\Si$ is a $C^{1,\al}\cap W^{2,2}$ surface away from the bad points $\xi_1,...,\xi_P$.\\
	Now we improve the regularity of $\Si$ up to $C^\infty$ around points contained in $\Om$ and different from the bad ones. This will be one of the main differences with \cite{SiEX} in the sense that the following argument only applies for $\xi\in\Om$. Locally parametrizing the surface with a function $w\in C^{1,\al}\cap W^{2,2}$ as before, we have that $w$ is a critical point for the functional $\int |A|^2 -\La$ on the domain of $w$. This implies that the first variation of the functional calculated on $w$ vanishes, that is:
	\begin{displaymath}
	\de \bigg( \int_{{\rm graph}(w) } |A_w|^2-\La\,d\HH^2 \bigg)=\de \bigg( \int_{dmn(w)} \sum_{ i,j,r,s=1}^2 (1-h)g^{ij}g^{rs}w_{ir}w_{js}\sqrt{g} -\La\sqrt{g}  \bigg)=0,
	\end{displaymath}
	where $dmn(w)$ denotes the domain of $w$. This relation is equivalent to say that $w$ satisfies in the weak sense a fourth order partial differential equation of the form:
	\begin{equation}\label{eqellittica}
	D_jD_s(A^{ijrs}(x,w,Dw)D_iD_rw)+D_jC^j(x,w,Dw,D^2w)+B^0(x,w,Dw,D^2w)=0,
	\end{equation}
	where:
	\begin{displaymath}
	C^j=B^j+\tB^j,
	\end{displaymath}
	with $A^{ijrs},B^j,B^0$ the coefficients given by the first variation of the functional $\int |A|^2$ and $\tB^j$ the ones coming from the first variation of $-\La\int_{ dmn(w)} \sqrt{g}$. That is:
	\begin{displaymath}
	\tB^j(x,z,p,q)= \tB^j(p)=  \La\frac{p_j}{\sqrt{1+\sum_ip_i^2}}.
	\end{displaymath}
	We know by \cite{SiEX} (page 310) that the coefficients $A^{ijrs},B^j,B^0$ satisfy the hypotheses of Lemma \ref{elliptic}. By a simple calculation also the coefficients $\tB^j$ satisfy the same relations, then we can apply Lemma \ref{elliptic} to get $w\in C^{2,\al}$. Hence by a bootstrap argument on Equation \eqref{eqellittica} we conclude that $w\in C^\infty$.\\
	
	\noindent At this point we know that $supp(V)=\Si\sqcup\{\xi_1,...,\xi_P \}$, with $\Si$ that is a $C^{1,\al}\cap W^{2,2}$ surface (and $C^\infty$ in $\Om\sm \{\xi_1,...,\xi_P \}$). From now on we will rename $\Si$ the union $\Si\sqcup\{\xi_1,...,\xi_P \}$, so that $supp(V)=\Si$, keeping in mind that the regularity of $\Si$ is achieved away from the bad points.\\
	
	\noindent Since we chose $\La<\LO$ we know that $C_\La>0$ and thus we can assume that $\Si_n$ is connected for any $n$. Together with boundedness of the Willmore energy, this implies that the sets $\Si_n$ converge to $\Si$ in the Hausdorff distance $d_\HH$ (see \cite{SiEX} page 310-311, or Theorem 3.4 in \cite{NoPo19} for a detailed proof), and hence in particular we get that $\Si$ is connected.\\
	Now we are going to derive the regularity also in neighborhoods of the bad points.  
	By the very same arguments of \cite{SiEX}, pages 313-316, one can find distinct points $y_1,...,y_{M+P}\in\Si$ with $y_{M+i}=\xi_i$ for $i=1,...,P$ and radii $\tau_k$ for $k=1,...,M+P$ such that $\Si\con\bigcup_{k=1}^{M+P} B_{\tau_k}(y_k)$ and for each $k\neq l$ we have that $\pa B_{\tau_k}(y_k)\cap \Si$ and $\pa B_{\tau_l}(y_l)\cap \Si$ are either disjoint or intersect transversely and:
	\begin{displaymath}
	\pa B_{\tau_k}(y_k)\cap\pa B_{\tau_l}(y_l)\cap\pa B_{\tau_m}(y_m)\cap\Si=\emptyset
	\end{displaymath}
	for distinct $k,l,m$.
	Moreover the curves $\Ga_l:=\bigg( \Si\sm \bigg( \bigcup_{k=M+1}^{M+P} B_{\tau_k}(y_k) \bigg) \bigg)\cap \pa B_{\tau_l} (y_l)$ for $l=1,...,M+P$ divide $\Si\sm ( \cup_{k=M+1}^{M+P} B_{\tau_k}(y_k) )$ into polygonal regions $R_1,...,R_Q$. And letting for all $l=1,...,Q$
	\begin{displaymath}
	\RR_l:=\bigg\{ x+z|x\in R_l,z\in (T_x\Si)^\perp,|z|\le\te\frac{\de}{4}  \bigg\},
	\end{displaymath}
	for some $\de$, it turns out that $\Si_n\cap \RR_l$ is $C^{1,\al}$-diffeomorphic to $R_l$ and then
	\begin{displaymath}
	\Si_n\sm \bigg( \bigcup_{k=M+1}^{M+P} B_{\tau_k}(y_k) \bigg)
	\end{displaymath}
	is $C^{1,\al}$-diffeomorphic to $	\Si\sm \bigg( \bigcup_{k=M+1}^{M+P} B_{\tau_k}(y_k) \bigg)$ for $n$ big enough (up to subsequence).\\
	So we can now take surfaces $\tS_n$ such that:
	\begin{displaymath}
	\tS_n\sm \bigg( \bigcup_{k=M+1}^{M+P} B_{\tau_k}(y_k) \bigg),\qquad 	\Si_n\sm \bigg( \bigcup_{k=M+1}^{M+P} B_{\tau_k}(y_k) \bigg), \qquad 	\Si\sm \bigg( \bigcup_{k=M+1}^{M+P} B_{\tau_k}(y_k) \bigg)
	\end{displaymath}
	are $C^{1,\al}$-diffeomorphic for all $n$, and:
	\begin{equation}\label{ref15}
	\tS_n\cap V_k=\Si_n\cap V_k \qquad \forall k=M+1,...,M+P
	\end{equation}
	for some neighborhood $V_k$ such that $B_{\tau_k}(y_k)\con\con V_k \con\con B_{2\tau_k}(y_k)$, and:
	\begin{equation} \label{ref16}
	\tS_n\sm \bigg( \bigcup_{k=M+1}^{M+P} B_{2\tau_k}(y_k) \bigg)=	\Si\sm \bigg( \bigcup_{k=M+1}^{M+P} B_{2\tau_k}(y_k) \bigg) \qquad\forall n,
	\end{equation}
	and also:
	\begin{equation}\label{ref17}
	\int_{ \tS_n\cap (B_{2\tau_k}(y_k)\sm B_{\tau_k}(y_k)) } |\tA_n|^2 \le C\ep^2,
	\end{equation}
	where $\tA_n$ is the second fundamental form of $\tS_n$.\\
	By \eqref{ref15} and the minimizing property of $\Si_n$ we have:
	\begin{displaymath}
	\begin{split}
	\int_{ \Si_n \sm ( \cup_{k=M+1}^{M+P} B_{\tau_k}(y_k) ) } |H_n|^2 &\le \int_{ \tS_n \sm ( \cup_{k=M+1}^{M+P} B_{\tau_k}(y_k) )} |\tH_n|^2+\La |\Si_n\sm ( \cup_{k=M+1}^{M+P} B_{\tau_k}(y_k) )|+\\&\indent-\La|\tS_n\sm ( \cup_{k=M+1}^{M+P} B_{\tau_k}(y_k) )|+\ep_n,
	\end{split}
	\end{displaymath}
	with $\ep_n\to 0$. Then by \eqref{ref16} and \eqref{ref17} we obtain:
	\begin{displaymath}
	\begin{split}
	\int_{ \Si_n \sm ( \cup_{k=M+1}^{M+P} B_{\tau_k}(y_k) ) } |H_n|^2 &\le \int_{ \Si \sm ( \cup_{k=M+1}^{M+P} B_{2\tau_k}(y_k) )} |H|^2+\ep_n+C\ep^2+\\&\indent+\La\big( |\Si_n\sm ( \cup_{k=M+1}^{M+P} B_{\tau_k}(y_k) )|-|\tS_n\sm ( \cup_{k=M+1}^{M+P} B_{\tau_k}(y_k) )|\big).
	\end{split}
	\end{displaymath}
	Using the hypotheses on the surfaces $\tS_n$, let us estimate the quantity:
	\begin{displaymath}
	\begin{split}
	|\Si_n\sm ( \cup_{k=M+1}^{M+P} B_{\tau_k}(y_k) )|&-|\tS_n\sm ( \cup_{k=M+1}^{M+P} B_{\tau_k}(y_k) )|\\&\le
	|\Si_n \sm ( \cup_{k=M+1}^{M+P} B_{2\tau_k}(y_k) )|-|\Si \sm ( \cup_{k=M+1}^{M+P} B_{2\tau_k}(y_k) )|+\\
	&\indent+\sum_{k=M+1}^{M+P} |\Si_n\cap B_{2\tau_k}(y_k)\sm ( B_{\tau_k}(y_k)\cup V_k) |+\\&\indent- | \cup_{k=M+1}^{M+P}  \tS_n\cap B_{2\tau_k}(y_k)\sm ( B_{\tau_k}(y_k)\cup V_k)   |  \le\\&\le
	|\Si_n \sm ( \cup_{k=M+1}^{M+P} B_{2\tau_k}(y_k) )|-|\Si \sm ( \cup_{k=M+1}^{M+P} B_{2\tau_k}(y_k) )|+\\
	&\indent+\sum_{k=M+1}^{M+P} |\Si_n\cap B_{2\tau_k}(y_k)\sm ( B_{\tau_k}(y_k)\cup V_k) | \le\\&\le
	|\Si_n \sm ( \cup_{k=M+1}^{M+P} B_{2\tau_k}(y_k) )|-|\Si \sm ( \cup_{k=M+1}^{M+P} B_{2\tau_k}(y_k) )|+\\
	&\indent+\sum_{k=M+1}^{M+P}|\Si_n\cap B_{2\tau_k}(y_k)|.  	    
	\end{split}
	\end{displaymath}
	Since this is true for all $\ep>0$, we get:
	\begin{displaymath}
	\begin{split}
	\int_{ \Si_n \sm ( \cup_{k=M+1}^{M+P} B_{\tau_k}(y_k) ) } |H_n|^2 &\le  \int_{ \Si \sm ( \cup_{k=M+1}^{M+P} B_{2\tau_k}(y_k) )} |H|^2+\ep_n
	+\\&+\La  \bigg(  |\Si_n \sm ( \cup_{k=M+1}^{M+P} B_{2\tau_k}(y_k) )|-|\Si \sm ( \cup_{k=M+1}^{M+P} B_{2\tau_k}(y_k) )|+\\
	&+\sum_{k=M+1}^{M+P}|\Si_n\cap B_{2\tau_k}(y_k)|  \bigg).
	\end{split}
	\end{displaymath}
	Since we already know that by the convergence of varifolds we also have $|\Si_n|\to \bM(\Si)=|\Si|$, then we get:
	\begin{displaymath}
	\begin{split}
	\limsup_n & \int_{ \Si_n \sm ( \cup_{k=M+1}^{M+P} B_{\tau_k}(y_k) ) }  |H_n|^2 \le \int_{ \Si \sm ( \cup_{k=M+1}^{M+P} B_{2\tau_k}(y_k) )} |H|^2
	+\\&+\limsup_n \La  \bigg(  |\Si_n \sm ( \cup_{k=M+1}^{M+P} B_{2\tau_k}(y_k) )|-|\Si \sm ( \cup_{k=M+1}^{M+P} B_{2\tau_k}(y_k) )|+\\
	&+\sum_{k=M+1}^{M+P}|\Si_n\cap B_{2\tau_k}(y_k)|  \bigg)= \\&=
	\int_{ \Si \sm ( \cup_{k=M+1}^{M+P} B_{2\tau_k}(y_k) )} |H|^2
	+ \La \sum_{k=M+1}^{M+P} |\Si\cap B_{2\tau_k}(y_k)|.
	\end{split}
	\end{displaymath}
	Hence finally:
	\begin{equation}
	\lim_{\si\searrow0}\limsup_n \int_{ \Si_n \sm ( \cup_{k=M+1}^{M+P} B_{\si}(y_k) ) }  |H_n|^2 \le \lim_{\si\searrow0} \int_{ \Si \sm ( \cup_{k=M+1}^{M+P} B_{\si}(y_k) )} |H|^2=\int_{\Si}|H|^2.
	\end{equation}
	Combining this with the natural lower semicontinuity of the Willmore functional under varifold convergence, we establish that:
	\begin{equation}
	|H_n|^2\HH^2\res \Si_n \longrightarrow |H|^2\HH^2\res \Si
	\end{equation}
	as measures on the domain $\R^3\sm\{\xi_1,...,\xi_P \}$.\\
	Moreover, with the above notation, we have by the Gauss-Bonnet Theorem that:
	\begin{displaymath}
	\begin{split}
	\int_{ \tS_n \sm ( \cup_{k=M+1}^{M+P} B_{\tau_k}(y_k) )} &|\tH_n|^2 - \int_{ \Si_n \sm ( \cup_{k=M+1}^{M+P} B_{\tau_k}(y_k) )} |H_n|^2 \\&=\frac{1}{4} \bigg(
	\int_{ \tS_n \sm ( \cup_{k=M+1}^{M+P} B_{\tau_k}(y_k) )} |\tA_n|^2 - \int_{ \Si_n \sm ( \cup_{k=M+1}^{M+P} B_{\tau_k}(y_k) )} |A_n|^2\bigg), 
	\end{split}
	\end{displaymath}
	then same conclusions hold for the second fundamental form, that is:
	\begin{equation}
	\lim_{\si\searrow0}\limsup_n \int_{ \Si_n \sm ( \cup_{k=M+1}^{M+P} B_{\si}(y_k) ) }  |A_n|^2 \le \int_{\Si}|A|^2,
	\end{equation}
	and:
	\begin{equation}
	|A_n|^2\HH^2\res \Si_n \longrightarrow |A|^2\HH^2\res \Si
	\end{equation}
	as measures on the domain $\R^3\sm\{\xi_1,...,\xi_P \}$.\\
	Finally we prove the claimed regularity of the varifold $\Si$, that is regularity in the bad points. According to the above discussion, let us sum up some useful results. For each
	$\de>0$ sufficiently small there is $\si\in(\de/2,\de)$ such that:
	\begin{equation} \label{ref18}
	\limsup_n \int_{ \Si_n\cap (\cup_{i=1}^P B_{2\si}(\xi_i)\sm B_{\si}(\xi_i)) } |A_n|^2 \le \de^2,
	\end{equation}
	\begin{equation} \label{ref23}
	\Si_n\sm \bigg( \bigcup_{i=1}^P B_{\si}(\xi_i) \bigg) \mbox{ is $C^{1,\al}$-diffeomorphic to } \Si\sm \bigg( \bigcup_{i=1}^P B_{\si}(\xi_i) \bigg),
	\end{equation}
	\begin{equation}
	\bigg| \WW\bigg( \Si_n\sm \bigg( \bigcup_{i=1}^P B_{\si}(\xi_i) \bigg) \bigg) -\WW(\Si) \bigg| \le \de^2.
	\end{equation}
	In particular choosing appropriate
	$\de_n\searrow0$ and then $\si_n\in(\de_n/2,\de_n)$, for all $i=1,...,P$ we have:
	\begin{equation}\label{ref21}
	\lim_n \WW\bigg(\Si_n\sm\bigg( \bigcup_{i=1}^P B_{\si_n}(\xi_i) \bigg)\bigg)=\WW(\Si).
	\end{equation}
	By Equation \eqref{ref23} we have that for $\si$ small enough $\Si\cap  B_{2\si}(\xi_i)\sm B_{\si}(\xi_i)$ is $C^1$-close to an annulus $L_i\cap  B_{2\si}(\xi_i)\sm B_{\si}(\xi_i)$. Hence we can take a smooth compact surface $\tS$ such that, for suitable points $y_1,...,y_p\in\tS$ and sufficiently small $\si$, $\tS\sm (\cup_{i=1}^P B_\si(y_i))$ is $C^{1,\al}$-diffeomorphic to $\Si\sm (\cup_{i=1}^P B_\si(\xi_i))$ and such that, for $\si=\si_n$ as above small enough, it is possible to replace $\tS\cap B_{\si_n}(y_i)$ by a slight deformation of $\Si_n\cap B_{\si_n}(\xi_i)$ followed by a rigid motion to give $(\Si_n\cap B_{\si_n}(\xi_i))^*$ such that the surface
	\begin{displaymath}
	\tS_n:=\bigg( \tS \sm \bigg(\bigcup_{i=1}^P B_{\si_n}(y_i)\bigg) \bigg)\cup \bigg( \bigcup_{i=1}^P \bigg( \Si_n\cap B_{\si_n}(\xi_i) \bigg)^* \bigg)
	\end{displaymath}
	is $C^{1,\al}\cap W^{2,2}$ and
	\begin{equation}\label{ref20}
	\WL((\Si_n\cap B_{\si_n}(\xi_i))^*) \le \WL(\Si_n\cap B_{\si_n}(\xi_i)) +\ep_n \qquad \ep_n\searrow0.
	\end{equation}
	Using the minimizing property of $\Si_n$ and then \eqref{ref20}, we have:
	\begin{equation}\label{ref24}
	\begin{split}
	\WL(\Si_n)&=\WL\bigg( \Si_n \cap \bigg(\bigcup_i B_{\si_n}(\xi_i)\bigg) \bigg)+\WL\bigg( \Si_n \sm \bigg(\bigcup_i B_{\si_n}(\xi_i)\bigg) \bigg)\\
	&\le \WL(\tS_n)+\ep_n \\
	&\le \WL \bigg( \Si_n \cap \bigg(\bigcup_i B_{\si_n}(\xi_i)\bigg) \bigg)+\WL \bigg( \tS \sm \bigg(\bigcup_i B_{\si_n}(y_i)\bigg) \bigg) +(P+1)\ep_n.
	\end{split}
	\end{equation}
	Hence:
	\begin{displaymath}
	\WL\bigg( \Si_n \sm \bigg(\bigcup_i B_{\si_n}(\xi_i)\bigg) \bigg)\le \WL \bigg( \tS \sm \bigg(\bigcup_i B_{\si_n}(y_i)\bigg) \bigg) +(P+1)\ep_n,
	\end{displaymath}
	and by \eqref{ref21} we get:
	\begin{equation}\label{eqminimalita}
	\WL(\Si)\le\WL(\tS).
	\end{equation}
	Analogously, using Gauss-Bonnet Theorem on the first inequality in \eqref{ref24}, being $\Si_n$ and $\tS_n$ diffeomorphic, we find:
	\begin{equation} \label{ref22}
	\int_{\Si}\bigg( |A|^2-\La\bigg) \le \int_{ \tS}\bigg( |\tA|^2-\La\bigg).
	\end{equation}
	Constructing $\tS$ taking a small perturbation of $\Si$ (so that no bad points lie on $\pa\Om$) and replacing $\Si\sm B_\si(\xi_i)$ with the graph of the function given by Lemma \ref{comparison}, by \eqref{ref22} we get the estimate
	\begin{displaymath}
	\int_{ \Si \cap B_\ro (\xi_i)} |A|^2 \le c\ro^{\al}+\La(|\Si\cap B_\ro(\xi_i)|-|\tS\cap B_\ro(\xi_i)|)\le C\ro^\al,
	\end{displaymath}
	for sufficiently small $\ro$ for some $\al>0$. Hence actually:
	\begin{displaymath}
	\int_{ \Si\cap B_\ro(y) } |A|^2 \le C\ro^{\al}
	\end{displaymath}
	for $\ro$ small enough and for all $y\in\Si$, now bad points included. And by classical arguments similar to the ones applied above in the case of good points one can show that this imply that $\Si$ is a $C^{1,\al}\cap W^{2,2}$ surface globally (and $\Si\cap\Om$ is of class $C^\infty$). In particular, arguing by approximation, we have $\WL(\Si)\ge C_\La$ and then by lower semicontinuity $\WL(\Si)= C_\La$ and by the upper bound on the Willmore energy we also conclude that $\Si$ is embedded.
\end{proof}

\noindent We conclude this section with some observations on the proof of Theorem \ref{thmregularity}.
\begin{remark}[Smallness of $\La$]
	The fundamental hypothesis of Theorem \ref{thmregularity} is to take the weight $\La$ sufficiently small. Observe that if $C_\LO=0$ and if there exists a minimizing sequence $(\Si^\LO_n)$ for $\WW_\LO$ with equibounded areas $|\Si_n^\LO|$ (i.e. if $\W$ is coercive, by Corollary \ref{cor:coercive}), then in the proof of Theorem \ref{thmregularity} we can take
	\begin{equation} \label{Lapiccolo}
	\La< \frac{4\pi}{m(\LO)},
	\end{equation}
	with
	\begin{equation}
	m(\LO)=\lim_n |\Si_n^\LO|,
	\end{equation}
	This estimate is sufficient for completing the proof. Of course the value of $m(\LO)$ depends on the chosen sequence $\Si^\LO_n$. It is interesting to notice that in the case of $\Om=B_{\frac{1}{2}}$, where we know that the sphere $S_{\frac{1}{2}}$ of radius $1/2$ ia a minimizer for $\La_{B_{\frac12}}=4$, we have $m(4)=|S_{\frac{1}{2}}|=\pi$; hence the estimate \eqref{Lapiccolo} gives $\La<4$, that is precisely the critical parameter $\La_{B_{\frac{1}{2}}}=4$, so in this case the estimate is sharp, excluding only the limit case of $\WW_\LO$.\\
	It could be a future development to prove or disprove the convergence to an enough regular surface for greater parameters $\La$ and in particular for the critical value $\LO$, perhaps using the more modern theory of $\cite{RiVP}$.
\end{remark}
\begin{remark}[Regularity of the limit surface]
	We derived the existence of a globally $C^{1,\al}\cap W^{2,2}$ surface $\Si$ that it is actually $C^\infty$ inside $\Om$, so if we know that $\Si\con \Om$ then $\Si$ is actually a smooth surface and hence a classical solution of the Problem $(P)_{\Om,\La}$. Also, $\pa \Om$ is of class $C^2$ by hypothesis, so on each relatively open set $A\con( \Si\cap \pa \Om)$, the surface is actually $C^2$.\\
	However, we want to notice here that it is not obvious that $\Si$ is globally $C^2$. In fact the smoothness of the surface inside $\Om$ is obtained by the Elliptic Regularity Lemma \ref{elliptic} used on Equation \eqref{eqellittica}, that is an equation given by the first variation of a functional, so it is something like $\frac{d}{dt}F(w+t\vp)|_{t=0}=0$ for the appropriate functional $F$. While this calculation is possible inside $\Om$, on $\pa\Om$ this leads only to a variational inequality of the fourth order subject to an obstacle boundary condition (given by the boundary of $\Om$), for which the development of a regularity theory is quite more difficult. Very remarkable results are proved in \cite{CaFrOB}, where it is studied the variational inequality of the bilaplacian $\De^2$ subject to obstacle boundary conditions; here it is proved that in dimension 2 (that is also our case) the solution in $C^2$. Of course our case is different, since the elliptic operator is nonlinear (recall \eqref{eqellittica}), but it is likely that we could achieve the same conclusion, having then $\Si$ of class $C^2$ globally and $C^\infty$ inside $\Om$ (hence getting a classical minimizer for the variational problem). This can be another possible development of the work, having also an interest itself in the theory of regularity for elliptic problems.
\end{remark}

\appendix
\section{Appendix}

In the Appendix we collect the technical results used in the proof of Theorem \ref{thmregularity} and some basic facts about varifold theory.\\

\begin{lemma}[\cite{SiEX}]\label{teta}
	Let $\Si$ be a compact surface without boundary, let $B_\ro$ be an open ball such that $\pa B_\ro$ intersects $\Si$ transversely and $\Si\cap B_\ro$ contains disjoint subsets $\Si_1,\Si_2$ with $\Si_j\cap B_{\te\ro}\neq\emptyset$, $\pa \Si_j\con\pa B_\ro$ and $|\pa \Si_j|\le\be\ro$ for $j=1,2$, where $\te\in(0,\frac{1}{2})$ and $\be>0$. Then
	\begin{displaymath}
	\WW(\Si)\ge 8\pi-C\be\te,
	\end{displaymath}
	with $C$ independent of $\Si,\be,\te$.
\end{lemma}

\begin{lemma}[Selection Principle, \cite{SiEX}]\label{selectionprinciple}
	If $\de>0$, if $I\con \R$ is a bounded interval and if $A_j\con I$ is a measurable set with measure $\ge \de$ for each $j=1,2,...$, then there exists a set $S\con I$ of measure $\ge\de$ such that each $x\in S$ lies in $A_j$ for infinitely many $j$.
\end{lemma}	

\noindent Here we have the results leading to the decay estimate \eqref{stimapsi}. Results of this kind are standard, however usually stated under more general forms; since we needed only the following more simple decay estimates, we prove such inequalities here for the convenience of the reader.

\begin{lemma} \label{lemmadecad}
	Let $f:(0,x_0]\to[0,+\infty)$ such that $f(x_0)>0$, $f$ is non decreasing and:
	\begin{displaymath}
	f\bigg( \frac{x}{2}  \bigg)\le \ga f(x)
	\end{displaymath}
	for all $x\in(0,x_0]$ for some $\ga\in(0,1)$. Then there are $C>0,\be\in(0,1)$ such that:
	\begin{displaymath}
	f(x)\le C\bigg(\frac{x}{x_0} \bigg)^\be f(x_0)
	\end{displaymath}
	for all $x\in(0,x_0]$.
\end{lemma}
\begin{proof}
	For all $x\in(0,x_0/2]$ there is $n$ such that $2^nx:=x'\in(x_0/2,x_0]$. Then:
	\begin{displaymath}
	\begin{split}
	f(x)&\le \ga^n f(x')=2^{-n\log_2(1/\ga)}f(x')=\bigg( \frac{x}{x'} \bigg)^{\log_2(1/\ga)}f(x')\\&=\bigg( \frac{x}{x_0} \bigg)^{\log_2(1/\ga)}\bigg( \frac{x_0}{x'} \bigg)^{\log_2(1/\ga)}f(x')\le \frac{1}{\ga}\bigg( \frac{x}{x_0} \bigg)^{\log_2(1/\ga)} f(x_0).
	\end{split}
	\end{displaymath}
	For all $x\in(x_0/2,x_0]$:
	\begin{displaymath}
	f(x)=\frac{1}{\ga}\bigg(\frac{1}{2}  \bigg)^{\log_2(1/\ga)}f(x)\le \frac{1}{\ga}\bigg(\frac{x}{x_0}  \bigg)^{\log_2(1/\ga)}f(x_0).
	\end{displaymath}
	Now if $\log_2(1/\ga)<1$ we are done, otherwise, since $x/x_0\le1$ for all $x$, we can choose an arbitrary $\be\in(0,1)$ and we have $(x/x_0)^{\log_2(1/\ga)}\le (x/x_0)^\be$.
\end{proof}

\begin{cor}\label{cordecad}
	Let $f:(0,x_0]\to[0,+\infty)$ such that $f(x_0)>0$, $f$ is non decreasing and:
	\begin{displaymath}
	f\bigg( \frac{x}{2}  \bigg)\le \ga f(x)+\al x^2
	\end{displaymath}
	for all $x\in(0,x_0]$ for some $\ga\in(1/2,1),\al\in(0,1/8)$. Then there are $C>0,\be\in(0,1)$ such that:
	\begin{displaymath}
	f(x)\le C\bigg(\frac{x}{x_0} \bigg)^\be f(x_0)
	\end{displaymath}
	for all $x\in(0,x_0]$.
\end{cor}

\begin{proof}
	Let $h(x)=f(x)+x^2$. We have:
	\begin{displaymath}
	h\bigg(\frac{x}{2} \bigg)=f\bigg(\frac{x}{2} \bigg)+\frac{x^2}{4}\le \ga f(x)+\bigg(\al+\frac{1}{4} \bigg)x^2\le \ga h(x).
	\end{displaymath}
	Applying Lemma \ref{lemmadecad} and taking $a>0$ such that $x_0^2\le af(x_0)$ we obtain:
	\begin{displaymath}
	f(x)\le h(x)\le K\bigg(\frac{x}{x_0}\bigg)^\be h(x_0)= K\bigg(\frac{x}{x_0}\bigg)^\be (f(x_0)+x_0^2)\le C\bigg(\frac{x}{x_0} \bigg)^\be f(x_0),
	\end{displaymath}
	with $C=K(1+a)$.
\end{proof}

%

\begin{lemma}[Elliptic Regularity, \cite{SiEX}] \label{elliptic}
	Let $\be,\ga,L>0$, $B^2=\{ x\in\R^2:|x|<1 \}$ and let
	\begin{displaymath}
	u=(u^1,...,u^m)\in W^{2,2}(B^2;\R^m)\cap C^{1,\ga}(B^2;\R^m)
	\end{displaymath}
	be such that $|u|+|Du|\le 1$ and:
	\begin{displaymath}
	\int_{ B^2\cap \{x:|x-\xi|<\ro \} } |D^2u|^2\le \be\ro^{2\ga}
	\end{displaymath}
	for each $\xi\in B^2$ and $\ro<1$. Moreover suppose that $u$ is a weak solution of the system:
	\begin{displaymath}
	D_jD_s(A^{ijrs}_{\al\be}(x,u,Du)D_iD_ru^\be)+D_jB^j_\al(x,u,Du,D^2u)+B^0_\al(x,u,Du,D^2u)=0
	\end{displaymath}
	where $A^{ijrs}_{\al\be}=A^{ijrs}_{\al\be}(x,z,p)$ and $B^j_\al=B^j_\al(x,z,p,q)$ satisfy:
	\begin{displaymath}
	\begin{split}
	\sum_{i,j,r,s,\al,\be} A^{ijrs}_{\al\be}\xi^\al_{ij}\xi^\be_{rs}&\ge L^{-1}\sum_{i,j,\al} |\xi^\al_{ij}|^2,\\
	|A^{ijrs}_{\al\be}(x,z,p)|\le L,\qquad& |D_{(x,z,p)}A^{ijrs}_{\al\be}(x,z,p)|\le L,\\
	|B^j_\al(x,z,p,q)|+|D_{(x,z,p)}&B^j_\al(x,z,p,q)|\le L(1+|q|^2),\\
	|D_q B^j_\al(x,z,p,q)|&\le L(1+|q|),
	\end{split}
	\end{displaymath}
	for all $|z|+|p|\le 1$ where $D_P F$ means the tensor of all first derivatives with respect to the variables $P$.\\
	Then $u\in W^{3,2}_{loc}(B^2)\cap C^{2,\al}$.
\end{lemma}

\noindent Finally, we list some facts about theory of varifolds that we used in the work.

\begin{thm}[Compactness of Varifolds, \cite{Al} and \cite{SiGMT}] \label{compactnessofvarifolds}
	Let $V_n=\bv(M_n,\te_n)$ be a sequence of 2-rectifiable varifolds in $U\con\R^3$ open such that:
	\begin{displaymath}
	(1) \qquad\sup_n \mu_{V_n}(W)+||\de V_n||(W) <+\infty \quad \forall W\con\con U,
	\end{displaymath}
	\begin{displaymath}
	(2) \qquad \exists \Te(V_n,x)\ge 1 \mbox{ on }U\sm A_n \, :\, \mu_{V_n}(A_n\cap W)\to0\quad \forall W\con\con U,
	\end{displaymath}
	where $\mu_V$ denotes the Radon measure on $U$ induced by a varifold $V$ and $\de V$ is its first variation, and where $\Te(V,x):=\lim_{r\searrow0}\frac{\mu_V(B_r(x))}{\pi r^2}$. Then there exists a subsequence $V_{n_k}$ converging to a rectifiable varifold $V$ with locally bounded first variation with the properties that:
	\begin{displaymath}
	\exists \Te(\mu_V,x)\ge 1 \quad \mu_V\mbox{-}ae \mbox{ in }U,
	\end{displaymath}
	\begin{displaymath}
	\liminf_n ||\de V_n||(W)\ge ||\de V||(W) \quad \forall W\con\con U.
	\end{displaymath}
	Moreover if each $V_{n_k}$ is integer, then $V$ is integer too.
\end{thm}
\begin{remark}
	It is very important to observe that if in Theorem \ref{compactnessofvarifolds} the varifolds $V_n$ are integer, then the hypothesis (2) is automatically satisfied (with sets such that $\mu_{V_n}(A_n)=0$).
\end{remark}
\noindent Also, we remind the concept of $\bF$-metric (\cite{PiEX}, page 66) used in the proof of Theorem \ref{thmbehaviour}, defined as follows.
\begin{defn}
	The $\bF$\emph{-metric} on $\VV_2(U)$, that is the set of 2-rectifiable integer varifolds with support contained in the open $U\con\R^3$, is defined as:
	\begin{equation}
	\bF(V,W)=\sup\{V(f)-W(f):f\in C_c(G_n(\R^{n+k})) ,|f|\le1,Lip(f)\le1  \}.
	\end{equation}
\end{defn}
\noindent And we have the useful:
\begin{lemma}[\cite{PiEX}, page 66]\label{Fmetricvarifold}
	In sets $\VV_2(U)\cap\{V:\bM(V)\le C<+\infty\}$ with $U\con\R^{n+k}$ open, the convergence of varifolds is equivalent to the convergence in the $\bF$-metric.
\end{lemma}



\begin{thebibliography}{}
	
	\addcontentsline{toc}{section}{References}
	
	\bibitem{Al} Allard W. : \emph{On the first variation of a varifold}, Ann. of Math. 95 (1972), 417-491.
	
	\bibitem{BaKu} Bauer M., Kuwert E. : \emph{Existence of minimizing Willmore surfaces of prescribed genus}, International Mathematics Research Notices 10 (2003), 553-576.
	
	\bibitem{BeMu04} Bellettini G., Mugnai L. : \emph{Characterization and representation of the lower semicontinuous envelope of the elastica functional}, Ann. Inst. H. Poincar\'{e}, Anal. Non Lin\'{e}aire 21(6) (2004) 839-880.
	%
	\bibitem{BeMu07} Bellettini G., Mugnai L. : \emph{A Varifolds Representation of the Relaxed Elastica Functional}, Journal of Convex Analysis Volume 14 (2007), No. 3, 543-564.
	
	\bibitem{CaFrOB} Caffarelli L.A., Friedman A. : \emph{The obstacle problem for the biharmonic operator}, Annali della Scuola Normale Superiore di Pisa 6 (1979), 151-184.
	
	\bibitem{DMN} Dayrens F., Masnou S., Novaga M. : \emph{Existence, regularity and structure of confined elasticae}, ESAIM: COCV 24 (2018), 25-43.
	
	\bibitem{DMR} Dondl P.W., Mugnai L., R\"{o}ger M. : \emph{Confined elastic curves}, SIAM J. Appl. Math., 71 (6), 2205-2226, 2011.
	
	\bibitem{EvGaMT} Evans L.C., Gariepy R.F. : \emph{Measure Theory and Fine Properties of Functions}, Revised Edition, Taylor and Francis Group (2015).
	
	\bibitem{GiTrEPDE} Gilbarg D., Trudinger N.S. : \emph{Elliptic Partial Differential Equations of Second Order}, Springer-Verlag Berlin Heidelberg 3$^{\mbox{\tiny rd}}$ edition (2001).
	
	\bibitem{Hu} Hutchinson J. : \emph{Second fundamental form for varifolds and the existence of surfaces minimizing curvature}, Indiana University Math. Journal 35 (1986), 45-71.
	
	\bibitem{Ku} Kusner R. : \emph{Comparison surfaces for the Willmore problem}, International Mathematics Research Notices 10 (2003).
	
	\bibitem{KeMoRi} Keller L.G.A., Mondino A., Rivi\`{e}re T. : \emph{Embedded surfaces of arbitrary genus minimizing the Willmore energy under isoperimetric constraint}, Arch. Ration. Mech. Anal. 212 (2014), 645-682.
	
	\bibitem{KuLiSc} Kuwert E., Li Y., Sch\"{a}tzle R. : \emph{The large genus limit of the infimum of the Willmore energy}, American Journal of Mathematics 132 (2010), 37-51.
	
	\bibitem{KuSc} Kuwert E., Sch\"{a}tzle R. : \emph{Removability of point singularities of Willmore surfaces}, Annals of Mathematics 160 (2004), 315-357.
	
	\bibitem{LeCP} Leonardi G.P. : \emph{An overview over the Cheeger Problem}, Preprint, http://cvgmt.sns.it/paper/2676/ , (2015).
	
	
	\bibitem{Ma} Mantegazza C. : \emph{Curvature varifolds with boundary}, Journal of Differential Geometry 43 (1996), 807-843.
	
	\bibitem{MaNeWC} Marques F.C., Neves A. : \emph{Min-Max theory and the Willmore Conjecture}, Annals of Mathematics 179 (2014), 683-782.
	
	\bibitem{Mi} Minicozzi W.P. : \emph{The Willmore functional on lagrangian tori: its relation to area and existence of smooth minimizers}, Journal of the American Mathematical Society 8 (1995), 761-791.
	
	\bibitem{MoRi} Mondino A., Rivi\`{e}re T. : \emph{Willmore spheres in compact riemannian manifolds}, Adv. Math. 232 (2013), 608-676.
	
	\bibitem{MuRoCS} M\"{u}ller S., R\"{o}ger M. : \emph{Confined structures of least bending energy}, Journal of Differential Geometry 97 (2014), 109-139.
	
	\bibitem{NoPo19} Novaga M., Pozzetta M. : \emph{Connected surfaces with boundary minimizing the Willmore energy}, Mathematics in Engineering, 2020, 2(3): 527-556.
	
	\bibitem{PhLT} Phillips R. : \emph{Liouville's Theorem}, Pacific Journal of Mathematics 28 (1969), 397-405.
	
	\bibitem{PiEX} Pitts J. : \emph{Existence and Regularity of Minimal Surfaces on Riemannian Manifolds}, Mathematical Notes, Princeton University Press (1981).
	
	\bibitem{Pozz} Pozzetta M. : \emph{On the Willmore Functional: Classical Results and New Extensions}, Master Degree Thesis \href{https://etd.adm.unipi.it/theses/available/etd-06022017-192024/}{(\underline{Link})}, Universit\`{a} di Pisa (2017).
	
	\bibitem{Po20} Pozzetta M. : \emph{A varifold perspective on the $p$-elastic energy of planar sets}, Journal of Convex Analysis 27 (2020) 845-879.
	
	\bibitem{RiAA} Rivi\`{e}re T. : \emph{Analysis aspects of Willmore surfaces}, Invent. math. 174 (2008), 1-45.
	
	\bibitem{RiLI} Rivi\`{e}re T. : \emph{Lipschitz conformal immersions from degenerating Riemann surfaces with $L^2$-bounded second fundamental forms}, Adv. Calc.Var. 6 (2013), 1-31.
	
	\bibitem{RiVP} Rivi\`{e}re T. : \emph{Variational principles for immersed surfaces with $L^2$-bounded second fundamental form}, Journal f\"{u}r die reine und angewandte Mathematik 695 (2014), 41-98.
	
	\bibitem{ScLSC} Sch\"{a}tzle R. : \emph{Lower semicontinuity of the Willmore functional for currents}, Journal of Differential Geometry 81 (2009), 437-456.
	
	\bibitem{ScWBP} Sch\"{a}tzle R. : \emph{The Willmore Boundary Problem}, Calc. Var 37 (2010), 275-302.
	
	\bibitem{Schy} Schygulla J. : \emph{Willmore minimizers with prescribed isoperimetric ratio}, Arch. Ration. Mech. Anal. 203 (2012), 901-941.
	
	\bibitem{SeCMV} Seifert U. : \emph{Configurations of fluid membranes and vesicles}, Advances in Physics 46 (1997), 13-137.
	
	\bibitem{SiGMT} Simon L. : \emph{Lectures on Geometric Measure Theory}, Proceedings of the Centre for Mathematical Analysis of Australian Nationa University (1984).
	
	\bibitem{SiEX} Simon L. : \emph{Existence of surfaces minimizing the Willmore functional}, Communications in Analysis and Geometry 1 (1993), 281-326.
	
	\bibitem{Top} Topping P. : \emph{Mean Curvature Flow and Geometric Inequalities}, J. Reine Angew. Math. 503 (1998) 47-61.
	
	\bibitem{Wi65} Willmore T.J. : \emph{Note on embedded surfaces}, Annals of Alexandru Cuza University, Section I, 11B (1965), 493-496.
	
	\bibitem{WiRG} Willmore T.J. : \emph{Riemannian Geometry}, Oxford Science Publications (1993).
	
	\bibitem{Wo} Wojtowytsch S. : \emph{Helfrich's energy and constrained minimisation}, Communications in Mathematical Sciences
	Volume 15, Number 8 (2017) 2373-2386.
	
\end{thebibliography}
\end{document}